\documentclass[a4paper,11pt,leqno]{amsart}
\usepackage{amsmath}
\usepackage{amsthm}
\usepackage{amsfonts}
\usepackage{amssymb}
\usepackage{texdraw}
\usepackage{mathbbol}
\usepackage{txfonts}
\usepackage{color}
\usepackage{mathrsfs}
\usepackage{verbatim}
\usepackage{stmaryrd}

\newtheorem{theorem}{Theorem}[section]{\bf}{\it}
\newtheorem{lemma}[theorem]{Lemma}{\bf}{\it}
\newtheorem{proposition}[theorem]{Proposition}{\bf}{\it}
\newtheorem{corollary}[theorem]{Corollary}{\bf}{\it}
\newtheorem{example}[theorem]{Example}{\bf}{\it}
{\bf}{\it} 
{\bf}{\it}
{\bf}{\it}

\newtheorem{defn}[theorem]{Definition}
\theoremstyle{remark}

\newtheorem{remark}[theorem]{Remark}

\numberwithin{equation}{section}

\newcommand{\R}{\mathbb R}

\newcommand{\N}{\mathbb N}

\newcommand{\loc}{{\operatorname{loc}}}

\newcommand{\capa}{\operatorname{cap}}

\newcommand{\spt}{\operatorname{spt}}

\newdimen\vintkern\vintkern11pt
\def\vint{-\kern-\vintkern\int}

\newcommand{\haus}{\mathcal{H}}

\newcommand{\dd}[1]{\;\mathrm{d} \| #1 \|}

\newcommand{\cW}{\mathcal{W}}
\renewcommand{\span}{\mathrm{span}}

\newcommand{\Lipconst}{\operatorname{Lip}} 

\newcommand{\poly}[1]{\mathcal{P}_{#1}}   
\newcommand{\polypar}[1]{\mathcal{P}_{#1}^+}   
\newcommand{\flatchain}{{\mathscr F}} 

\newcommand{\mass}{{\mathbf M}} 
\newcommand{\bdry}{{\partial}} 

\newcommand{\dblebracket}[1]{\llbracket #1 \rrbracket}  

\newcommand{\rstr}{\:\mbox{\rule{0.1ex}{1.2ex}\rule{1.1ex}{0.1ex}}\:} 

\newcommand{\lm}{{\mathscr L}} 

\begin{document}

\title{Wolfe's theorem for weakly differentiable cochains}
\author{Camille Petit \and Kai Rajala}
\author{Stefan Wenger}
\address{C.P. and K.R.: University of Jyv\"askyl\"a, Department of Mathematics and Statistics (P.O. Box 35), FI-40014 University of Jyv\"askyl\"a, Finland}
\address{S.W.: Universit\' e de Fribourg, Math\' ematiques, Ch. du Mus\' ee 23, 1700 Fribourg, Switzerland}
\thanks{C.P. and K.R. were supported by the Academy of Finland, project $\#$257482. Parts of this research were carried out when K.R. was visiting Universit\'e de Fribourg. He would like to thank the department for its hospitality. Parts of this research were carried out when S.W. was visiting the Department of Mathematics and Statistics at the University of Jyv\"askyl\"a. He would like to thank the department for its hospitality.}

\begin{abstract}
A fundamental theorem of Wolfe isometrically identifies the space of flat differential forms of dimension $m$ in $\R^n$ with the space of flat $m$-cochains, that is, the dual space of flat chains of dimension $m$ in $\R^n$. The main purpose of the present paper is to generalize Wolfe's theorem to the setting of Sobolev differential forms and Sobolev cochains in $\R^n$. A suitable theory of Sobolev cochains has recently been initiated by the second and third author. It is based on the concept of upper norm and upper gradient of a cochain, introduced in analogy with Heinonen-Koskela's concept of upper gradient of a function.
\end{abstract}

\maketitle


\section{Introduction}

In the 1940's, Whitney initiated a geometric integration theory, see \cite{Whi57}, the purpose of which was to integrate a quantity over ``$m$-dimensional sets" in such a way that the integral depends on the position of the set in $\R^n$. The quantities one integrates are $m$-dimensional flat forms and the sets over which one integrates are $m$-dimensional flat chains. Flat forms are $L^\infty$-differential forms with $L^\infty$-exterior derivatives. The flat norm of such a form is defined as the maximum of the $L^\infty$-norm of the form and that of its derivative. In order to define flat chains, one first considers the space $\poly{m}= \poly m(\R^n)$ of polyhedral chains in $\R^n$, that is finite formal sums of oriented $m$-dimensional polyhedra in $\R^n$ with real multiplicities, see Section \ref{sect:preliminaries} for precise definitions. One equips $\poly m$ with the flat norm, given for $T\in\poly m$ by
$$|T|_\flat := \inf \{ \mass(R)+\mass(S): R\in\poly{m}, S\in\poly{m+1}, R+\bdry S=T\},$$
where $\bdry S$ is the boundary of $S$ and $\mass(R)$ is the mass of $R$. The completion of $\poly m$ under the flat norm is called the space of flat $m$-chains in $\R^n$ and denoted $\flatchain_m$. Elements of the dual space of $\flatchain_m$ are called flat $m$-cochains.
In \cite{Wol48}, see also \cite{Whi57}, Wolfe proved the following fundamental theorem: the space of flat $m$-forms, endowed with the flat norm, is isometric to the space of flat $m$-cochains.
Wolfe's theorem has recently been generalized to the setting of Banach spaces by Snipes in \cite{Sni13}, where she defines a flat partial differential form in a Banach space and shows that the space of these forms is isometrically the dual space of the space of flat chains as defined by Adams \cite{Ada08}. Moreover, Wolfe's theorem has recently been used by Heinonen-Sullivan \cite{HS02} and Heinonen-Keith \cite{HKe11}, see also Heinonen-Rickman \cite{HR02}, to give conditions under which a metric space is locally bi-Lipschitz equivalent to $\R^n$. 

The main purpose of the present paper is to generalize the classical Wolfe's theorem to the setting of Sobolev differential forms and Sobolev cochains in $\R^n$. A suitable theory, based on upper gradients, of Sobolev cochains in complete metric measure spaces has recently been initiated by the second and third authors in \cite{RW13}. Before stating our main results we briefly recall the relevant definitions from \cite{RW13}, restricting ourselves to the setting of $\R^n$. We refer to Section~\ref{sec:Sobolev-cochains} for precise definitions.
A subadditive $m$-cochain on $\poly m$ is a function $X: \poly m \to \overline \R$ which satisfies $X(0)=0$ and which is subadditive in the sense that
$$|X(T)|\leq |X(T+R)|+|X(R)|,$$
for all $T,R\in\poly m$. If furthermore $X(T+R)=X(T)+X(R)$ whenever each term is finite then $X$ is called a (additive) cochain. In \cite{RW13} a notion of upper gradient of a subadditive cochain is defined in analogy with Heinonen-Koskela's notion of upper gradient of a function \cite{HKST}.
 A Borel function $g:\R^n \to [0,\infty]$ is called upper gradient of $X$ if
$$|X(T)|\leq \int_{\R^n} g\ d\| S\|$$
for all $T\in\poly{m}$ and $S\in\poly{m+1}$ satisfying $\bdry S=T$, where $\| T\|$ denotes the mass measure of $T$, see Section~\ref{subsec:polyhedral chains}. Similarly, a Borel function $h:\R^n \to [0,\infty]$ is an upper norm of $X$ if 
$$|X(T)|\leq \int_{\R^n} h\ d\| T\|$$ for all $T\in\poly m$.
Given an additive cochain $X$ with upper norm in $L^q(\R^n)$ and upper gradient in $L^p(\R^n)$, we define its Sobolev norm by $$\| X\|_{q,p}=\max\left\{\inf  \|h\|_q, \inf\| g\|_p\right\},$$ where the infima are taken with respect to upper norms $h$ and upper gradients $g$ of $X$, respectively. Note that this norm is different from but equivalent to the norm introduced in \cite{RW13}. The Sobolev space $W_{q,p}(\poly{m})$ of additive cochains is the set of equivalence classes of additive cochains with upper norm in $L^q(\R^n)$ and upper gradient in $L^p(\R^n)$, under the equivalence relation defined by $X_1\sim X_2$ if $\|X_1-X_2\|_{q,p}=0$. 


The main result of the present paper is the following generalization of the classical Wolfe's theorem.


\begin{theorem}   \label{thm:main result}
Let $1\leq m\leq n$ and $1<q,p< \infty$. If $p> n-m$ or $q\leq \frac{pn}{n-p}$ then the space $W_d^{q,p}(\R^n,\bigwedge^m)$ of Sobolev differential forms is isometrically isomorphic to the space $W_{q,p}(\poly{m})$.
\end{theorem}

Recall that the space $W_d^{q,p}(\R^n,\bigwedge^m)$ of Sobolev differential $m$-forms consists of those $L^q$-integrable differential $m$-forms $\omega$ whose distributional exterior derivatives $d\omega$ are $L^p$-integrable. It is endowed with the norm $$\|\omega\|_{q,p}= \max\left\{\left(\int \|\omega(x)\|^qdx\right)^\frac{1}{q}, \left(\int \|d\omega(x)\|^pdx\right)^\frac{1}{p}\right\},$$ where $\|\omega(x)\|$ denotes the comass norm of $\omega(x)$. See Section \ref{sect:preliminaries} for the precise definitions. It is well-known that the Sobolev space $W^{1,p}$ of functions is (for all $p \geq 1$) isomorphic to the Newtonian Sobolev space $N^{1,p}$ defined using upper gradients, cf. \cite{HKST}. This remains true in the setting of Theorem \ref{thm:main result} when $m=0$, cf. \cite[Proposition 3.11]{RW13}. Theorem \ref{thm:main result} gives a partial answer to the question whether the same result holds for $m \geq 1$. 

The classical Wolfe's theorem is an easy consequence of Theorem~\ref{thm:main result}. 

\begin{corollary}\label{cor:Wolfe's theorem}
For $1\leq m \leq n$ the space $W_d^{\infty,\infty}(\R^n,\bigwedge^m)$ is isometrically isomorphic to the dual space of the space $\flatchain_m$ of flat chains in $\R^n$.
\end{corollary}

Here, $W_d^{\infty,\infty}(\R^n,\bigwedge^m)$ is endowed with the norm $\|\omega\|_{\infty,\infty}$ which is defined similarly to $\|\omega\|_{q,p}$ but using the essential supremum of the pointwise comass norms.

As mentioned above, Heinonen-Sullivan \cite{HS02} and Heinonen-Keith \cite{HKe11} applied Wolfe's theorem in order to give conditions under which a metric space is locally bi-Lipschitz equivalent to $\R^n$. Finding similar conditions for quasiconformal equivalence is an interesting open problem. To attack this problem, it is desirable to find generalizations of Wolfe's theorem for Sobolev forms. Theorem \ref{thm:main result} was partially motivated by this application.

We do not know whether Theorem \ref{thm:main result} holds for all values of $q$ and $p$. However, we have an unconditional result for cochains on $\poly{m}^0$, the space of polyhedral $m$-chains without boundary. In order to state our result, define the norm of a cochain $X$ on $\poly{m}^0$ with upper gradient in $L^p(\R^n)$ by $\| X\|_p=\inf \| g\|_p$, where the infimum is taken with respect to upper gradients $g$ of $X$. Let $W_p(\poly{m}^0)$ be the set of equivalence classes of additive cochains on $\poly{m}^0$ with  upper gradient in $L^p(\R^n)$, under the equivalence relation defined by $X_1\sim X_2$ if $\|X_1-X_2\|_p=0$. Denote furthermore by $\overline{W}_d^{p}(\R^n, \bigwedge^m)$ the quotient space of the space of $m$-forms on $\R^n$ with coefficients in $L^1_{\loc}(\R^n)$ and coefficients of the distributional exterior derivative in $L^p(\R^n)$ by the subspace of those elements $\omega$ with $d\omega = 0$. For $[\omega]\in \overline{W}_d^{p}(\R^n, \bigwedge^m)$ define $\|[\omega]\|_p:= \|d\omega\|_p$, where $\|\cdot\|_p$ denotes the $L^p$-norm of the pointwise comass norm. This defines a norm on $\overline{W}_d^{p}(\R^n, \bigwedge^m)$ which is bounded by the quotient norm. Our result can now be stated as follows.

\begin{theorem}\label{thm:main-result-closed}
For $1\leq m \leq n-1$ and $1<p<\infty$ the space $\overline{W}_d^{p}(\R^n, \bigwedge^m)$ is isometrically isomorphic to $W_p(\poly{m}^0)$.
\end{theorem}

We briefly outline the proofs of Theorems \ref{thm:main result} and \ref{thm:main-result-closed}. A smooth compactly supported differential form naturally induces a Sobolev cochain by integration. Using this observation and approximation of Sobolev forms by smooth forms, we show that there exists a linear, norm-preserving map mapping the space of Sobolev forms to the space of Sobolev cochains with corresponding exponents. On the other hand, we construct a Sobolev form from a Sobolev cochain as follows: we restrict the cochain to the $m$-planes induced by coordinate vectors, and then use Lebesgue differentiation to construct the coefficients of the resulting form. The map defined this way is also linear and norm-preserving. 
To prove Theorem \ref{thm:main result}, we show that the two maps are actually inverses to each other. The main problem in showing this is that it is difficult to see why the restriction of a non-zero Sobolev cochain to the coordinate $m$-planes should be non-zero. In other words, why should a non-zero Sobolev cochain induce a non-zero Sobolev form? To overcome this problem, we ``smoothen" the cochains by applying averages. 
Given a cochain $X$ on $\poly{m}$ and $r>0$ we set
$$X_r(T):= \fint_{B(0,r)} X({\varphi_x}_\#T) dx$$
for every $T \in\poly m$, where $\varphi_x:\R^n\to\R^n$ is the translation map $\varphi_x(y)=x+y$.  Here, $\fint_E$ denotes the integral average $\lm^n(E)^{-1}\int_E$ and $\lm^n$ denotes Lebesgue measure.  The integrand is measurable and locally integrable under our assumptions, see Lemma~\ref{lem:mb-loc-int-transl}. Using the Federer-Fleming deformation theorem, we show that the cochains $X_r$ are determined by their action in the coordinate $m$-planes. Theorem \ref{thm:main result} follows if we can show that the cochain $X$ can be approximated by the 
cochains $X_r$. This is given by the following continuity result, which is of independent interest.

\begin{theorem}\label{thm:cont-intro}
Let $0\leq m\leq n$ and $1<p,q\leq\infty$. If $X \in W_{q,p}(\poly{m})$ then
$$| X_r(T)-X(T)|\to 0 \text{ as } r\to 0$$
for every $T\in \poly{m}\setminus \Lambda$ for some family $\Lambda \subset \poly{m}$ of zero $\nu$-modulus, where $\nu=q$ if $p> n-m$ and
$$\nu=\min \{q,pn/(n-p)\}$$
otherwise.
\end{theorem}
Proving Theorem \ref{thm:main result} for all exponents $p$ and $q$ would require a stronger form of Theorem \ref{thm:cont-intro}. The modulus appearing in the statement measures the size of exceptional sets, see Section \ref{subsec:modulus} for the definition. 
A similar statement holds for cochains on $\poly{m}^0$, see Theorem~\ref{thm:smoothening}. This, together with the arguments above and considerations involving the so-called coboundary of a cochain, are used to prove Theorem \ref{thm:main-result-closed}. 

We finally mention that a different variant of Wolfe's theorem for Sobolev forms was given in \cite{GKS83}. In that paper, the authors provide a one-to-one correspondence between the forms in $W_d^{q,p}(\R^n,\bigwedge^m)$ and linear functionals on $\poly m$ which, together with their exterior derivatives, satisfy certain boundedness conditions with respect to a so-called $q$-mass and $p$-mass. We believe that the notion of Sobolev cochain used in the present paper is more natural than the one defined in \cite{GKS83}.

\bigskip

Our paper is structured as follows. In Section~\ref{sect:preliminaries}, we recall the definitions of Sobolev forms, polyhedral chains, and Sobolev cochains. In Section~\ref{sec:continuity}, which is the most substantial part of the paper, we prove the main continuity result, Theorem~\ref{thm:cont-intro}, and an analogous version for cochains on $\poly{m}^0$, see Theorem~\ref{thm:smoothening}. In Section~\ref{sec:form}, we construct a linear map from the space of Sobolev forms to the space of Sobolev cochains and show that this map preserves norms. 
In Section~\ref{sec:cochain}, we construct a continuous linear map from the space of Sobolev cochains to the space of Sobolev forms. Finally, Section~\ref{sec:proof-main-thm} contains the proofs of Theorems~\ref{thm:main result} and \ref{thm:main-result-closed}.


\section{Preliminaries}   \label{sect:preliminaries}
In this section we collect the definitions of the basic objects of the present paper.

\subsection{Sobolev differential forms in $\R^n$}
We recall the definition of $W_d^{q,p}(\R^n, \bigwedge^m)$, the space of (weak) Sobolev differential $m$-forms on $\R^n$. We refer to \cite{IL93} for details.
Let $1\leq m \leq n$ and set $$\Lambda(m,n):= \{ \alpha:\{ 1,\dots,m\} \to \{1,\dots , n\} \text{ strictly increasing}\}.$$ 
Let $\omega$ be an $m$-form on $\R^n$, given in Euclidean coordinates by $$\omega = \sum_{\alpha\in\Lambda(m,n)} \omega(\,\cdot\,,\alpha)\, dx^\alpha,$$ with locally integrable coefficients $\omega(\,\cdot\,, \alpha)$. An $(m+1)$-form $d\omega$ on $\R^n$, given in Euclidean coordinates by
$$d\omega = \sum_{\beta\in\Lambda(m+1,n)} d\omega(\,\cdot\,,\beta)\, dx^\beta$$
and with locally integrable coefficients $d\omega(\,\cdot\,,\beta)$, is said to be the distributional exterior derivative of $\omega$, if 
$$
\int_{\R^n} d\omega \wedge \nu = (-1)^{m+1} \int_{\R^n} \omega \wedge d\nu
$$
for every $C^{\infty}$-smooth compactly supported $(n-m-1)$-form $\nu$.
Given $1\leq p,q\leq\infty$, the space $W_d^{q,p}(\R^n,\bigwedge^m)$ consists of (equivalence classes of) $m$-forms $\omega$ on $\R^n$ with coefficients $\omega(\,\cdot\,, \alpha)$ in $L^q(\R^n)$ and such that $\omega$ has a distributional exterior derivative $d\omega$ with coefficients $d\omega(\,\cdot\,, \beta)$ in $L^p(\R^n)$. 

We endow $W_d^{q,p}(\R^n, \bigwedge^m)$ with the following norm, which is different but equivalent to the norm considered in \cite{IL93} and \cite{RW13}.
For this, denote by $|\cdot|$ the norm on the space $\bigwedge_m\R^n$ of $m$-vectors associated with the  inner product for which $\{e_{\alpha(1)}\wedge\dots\wedge e_{\alpha(m)}: \alpha\in\Lambda(m,n)\}$ is an orthonormal basis. Here and throughout the text, $e_j$ denotes the $j$-th standard unit vector in $\R^n$. An $m$-vector $\xi\in\bigwedge_m\R^n$ is called simple if it can be written in the form $\xi=\xi_1\wedge\dots\wedge \xi_m$ for vectors $\xi_i\in\R^n$, $i=1,\dots,m$. The comass of an $m$-covector $\nu\in\bigwedge^m\R^n$ is defined by $$\|\nu\|= \sup\{\langle \nu, \xi\rangle: \text{$\xi\in\bigwedge\nolimits_m\R^n$ simple, $|\xi|\leq 1$}\},$$
where $\langle \cdot,\cdot\rangle$ denotes the natural pairing of $m$-covectors and $m$-vectors. 
Given an $m$-form $\omega$ on $\R^n$, with coefficients in $L^q(\R^n)$, we set
$$\|\omega\|_q:= \left(\int \|\omega(x)\|^qdx\right)^\frac{1}{q}$$
if $q<\infty$, where $\|\omega(x)\|$ denotes the comass of $\omega(x)$. We define $\|\omega\|_q$ analogously in case $q=\infty$. The Sobolev norm of an element $\omega\in W_d^{q,p}(\R^n,\bigwedge^m)$ is then defined by
$$\|\omega\|_{q,p}= \max\left\{\|\omega\|_q,\|d\omega\|_p\right\}.$$

Finally, we denote by $W_{d,\loc}^{1,p}(\R^n, \bigwedge^m)$ the space of $m$-forms on $\R^n$ with coefficients in $L^1_{\loc}(\R^n)$ and coefficients of the distributional exterior derivative in $L^p(\R^n)$. The space $\overline{W}_d^{p}(\R^n, \bigwedge^m)$ is the quotient of $W_{d,\loc}^{1,p}(\R^n, \bigwedge^m)$ by the subspace of those elements $\omega$ with $d\omega = 0$. The norm of an element $[\omega]\in \overline{W}_d^{p}(\R^n, \bigwedge^m)$ is defined by
$$\|[\omega]\|_p:= \|d\omega\|_p.$$
This clearly defines a norm on $\overline{W}_d^{p}(\R^n, \bigwedge^m)$. Note that if $m=0$ and $f$ is a Sobolev function then $\|[f]\|_p = \| \,|\nabla f|\,\|_p$.

\subsection{Polyhedral chains in $\R^n$} \label{subsec:polyhedral chains}
We recall the basic definitions related to polyhedral chains. We refer to \cite{Whi57} and \cite{Hei05} for further details. 
Let $0\leq m\leq n$. Formally, a polyhedral $m$-chain $T$ in $\R^n$ is a formal finite  sum
\begin{equation}
\label{pohe}
T=\sum_{i=1}^N a_i T_i, 
\end{equation}
where $a_i\in\R$ and $T_i$ is an oriented $m$-dimensional polyhedron in $\R^n$ in case $m\geq 1$ and $T_i$ is a point in $\R^n$ in case $m=0$. More precisely, consider the additive group of formal finite sums (with coefficients in $\R$) of compact, convex, oriented $m$-dimensional polyhedra (respectively, points if $m=0$). Then, quotient by the equivalence relation identifying $-T$ with $\tilde T$, where $\tilde T$ is $T$ with the opposite orientation, and identifying $T$ with $T_1+T_2$ if $T$ is formed by gluing $T_1$ and $T_2$ along a face with the correct orientation. The quotient group is the set of \textit{polyhedral $m$-chains} in $\R^n$, which we denote by $\poly{m}(\R^n)$ or $\poly{m}$ for short.

It is easily seen that for every $T\in\poly{m}$ there exists a representation \eqref{pohe} such that the $T_i$ have non-overlapping interiors. We associate to a polyhedral $m$-chain $T\in\poly{m}$ a finite measure, denoted by $\|T\|$ and defined by
$$\|T\|:=\sum_{i=1}^N |a_i| {\haus^m}\rstr{T_i},$$
where $\sum a_i T_i$ is a non-overlapping representation of $T$. Here, $\haus^m$ denotes the $m$-dimensional Haudorff measure. The number $\| T\|(\R^n)$ is called the \textit{mass} of $T$ and denoted by $\mass(T)$.
It is worth mentioning that a polyhedral $m$-chain $T$ gives rise to an $m$-dimensional normal current in $\R^n$ by integrating smooth compactly supported $m$-forms over $T$, see \cite{Fed69}, and thus also to an $m$-dimensional metric current in the sense of \cite{AK00}. 

The \textit{boundary} $\bdry T$ of $T\in\poly{m}$, $1\leq m\leq n$, is a polyhedral $(m-1)$-chain defined in the usual way, namely the boundary of a polygon is the sum of its faces with the induced orientations. If $T=\sum a_i T_i$ is a polyhedral $0$-chain then we write $\bdry T=0$ if and only if $\sum a_i=0$. Note that we have $\bdry\bdry T=0$ for every $T\in\poly{m}$ with $m\geq 2$. Denote by $\poly{m}^0$ the set of polyhedral $m$-chains $T\in\poly m$ with $\bdry T=0$, and $\poly{m}^+$ the set of polyhedral $m$-chains $T$ such that for the non-overlapping representation $T=\sum a_i T_i$ mentioned above, each polyhedron $T_i$ is parallel to one of the $m$-dimensional coordinate planes. Note that if $T\in\poly{m}^+$ then, in general, $\bdry T$ need not be in $\poly{m-1}^+$.

If $T=\sum_{i=1}^N a_i T_i$ is a polyhedral $m$-chain and $\varphi$ is an affine map, then $\varphi_{\#}T$ is the polyhedral $m$-chain defined by $$\varphi_{\#}T = \sum_{i=1}^N a_i \varphi(T_i)$$ and is called the push-forward of $T$ by the map $\varphi$. If $x \in \R^n$, we use the notation $\varphi_x$ for the translation map $y \mapsto x+y$. The push-forward 
${\varphi_{x}}_\#T$ is thus simply the translate of $T$.

\subsection{Sobolev cochains in $\R^n$}\label{sec:Sobolev-cochains}
We recall the basic definitions from the theory of weakly differentiable cochains initiated by the second and third author in \cite{RW13}. Whereas the definitions in \cite{RW13} are given for arbitrary complete metric spaces $X$ and metric normal or integral currents in $X$ in the sense of Ambrosio-Kirchheim \cite{AK00} we will restrict ourselves to the setting of $\R^n$ and the spaces $\poly{m}$ and $\poly{m}^0$ of polyhedral chains in $\R^n$ in this paper. We will therefore only give the relevant definitions in this setting.


\begin{defn}\label{def:cochain}
Let $0\leq m\leq n$. A function $X:\poly{m}\to \overline\R$ is called \textit{a subadditive cochain on $\poly m$} if $X(0)=0$ and
$$|X(T)| \leq |X(T+R)| + |X(R)|$$
for all $T,R\in\poly m$. If furthermore
$$X(T+R)=X(T)+X(R)$$
whenever each term is finite, then $X$ is called \textit{an additive cochain}, or simply a cochain, on $\poly{m}$.
\end{defn}

Cochains on $\poly{m}^0$ are defined by simply replacing $\poly{m}$ by $\poly{m}^0$ everywhere in the definition above.

A large class of additive cochains comes from differential forms.

\begin{example}
Given a (smooth) differential $m$-form $\omega$ on $\R^n$, we can define a cochain $X^\omega:\poly{m}\to\overline\R$ by setting $X^\omega(T)=\int_T \omega$ for every $T\in\poly{m}$.
\end{example}

The following notions of upper norm and upper gradient of a subadditive cochain defined in \cite{RW13} are in analogy with the definition of upper gradient of a function.

\begin{defn}
Let $X$ be a subadditive cochain on $\poly{m}$. 
\begin{enumerate}
\item A Borel function $h:\R^n \to [0,\infty]$ is called \textit{upper norm of $X$} if
\begin{equation}   \label{eq:upper norm}
|X(T)|\leq \int_{\R^n} h \dd{T}
\end{equation}
for every $T\in \poly{m}$.
\item A Borel function $g:\R^n \to [0,\infty]$ is called \textit{upper gradient of $X$} if
\begin{equation}    \label{eq:upper gradient}
|X(\bdry S)|\leq \int_{\R^n} g \dd{S}
\end{equation}
for all $S\in\poly{m+1}$.
\end{enumerate}
\end{defn}

The upper norm and upper gradient of cochains on $\poly{m}^0$ are defined analogously. We note here that throughout this paper, when dealing with cochains on $\poly{m}^0$ we will only use upper gradients. 
In \cite{RW13}, the authors proved that upper gradients of $0$-cochains are exactly upper gradients of functions. 

For $1\leq p,q\leq\infty$ denote by $\cW_{q,p}(\poly{m})$ the set of subadditive $m$-cochains which have an upper norm in $L^q(\R^n)$ and an upper gradient in $L^p(\R^n)$. In \cite{RW13} the notation $\cW_{q,p}(\poly{m}, \poly{m+1})$ was used. We define the Sobolev norm of a subadditive cochain $X \in\cW_{q,p}(\poly{m})$ by $$\| X\|_{q,p}=\max\left\{\inf  \|h\|_q, \inf\| g\|_p\right\},$$ where the infima are taken with respect to upper norms $h$ and upper gradients $g$ of $X$, respectively. This norm is different from but equivalent to the norm introduced in \cite{RW13}. 

Given two (additive) cochains $X_1,X_2:\poly m\to \overline \R$ the cochain $X_1+X_2$ is defined by $(X_1+X_2)(T)=X_1(T)+X_2(T)$ if $|X_1(T)|+|X_2(T)| <\infty$, and $(X_1+X_2)(T)=\infty$ otherwise. The space $W_{q,p}(\poly{m})$ is then defined as the set of equivalence classes of additive cochains in $\cW_{q,p}(\poly{m})$ under the equivalence relation defined by $X_1\sim X_2$ if $\| X_1-X_2\|_{q,p}=0$. In \cite{RW13} the notation $W_{q,p}(\poly{m}, \poly{m+1})$ was used instead. It is clear that $W_{q,p}(\poly{m})$ is a vector space. Note that the classical space of flat cochains is isometrically isomorphic to $W_{\infty,\infty}(\poly{m})$, see Lemma \ref{lem:flat cochains}. 

For $1\leq p\leq \infty$ denote by $\cW_p(\poly{m}^0)$ the set of subadditive cochains on $\poly{m}^0$ which have an upper gradient in $L^p(\R^n)$. The Sobolev norm of an element $X \in\cW_p(\poly{m}^0)$ is defined by $$\| X\|_p=\inf\| g\|_p,$$ where the infimum is taken with respect to upper gradients $g$ of $X$. The  space $W_p(\poly{m}^0)$ is defined to be the set of equivalence classes of additive cochains in $\cW_p(\poly{m}^0)$ under the equivalence relation defined by $X_1\sim X_2$ if $\| X_1-X_2\|_p=0$.

The \textit{coboundary} $dX$ of a subadditive $m$-cochain $X$ is the subadditive $(m+1)$-cochain defined by $dX(S)=X(\bdry S)$ for all $S\in\poly{m+1}$. It follows from the definition that a Borel function is an upper gradient of $X$ if and only if it is an upper norm of $dX$. Note that for all $1\leq q,p\leq \infty$ and all $1\leq s\leq \infty$ the coboundary operator yields a linear map
$$d: W_{q,p}(\poly{m}) \to W_{p,s}(\poly{m+1}).$$

\subsection{Modulus and capacity for polyhedral chains} \label{subsec:modulus}
Let now $\Lambda \subset \poly{m}$ be a family of polyhedral $m$-chains and $1\leq p<\infty$. The \textit{$p$-modulus $M_p(\Lambda)$} is defined as $\inf \int_{\R^n} f^p d\lm^n$, where the infimum is taken over all non-negative Borel functions $f$ such that $\int_{\R^n} f \dd{T}\geq 1$ for every $T\in\Lambda$. The theory of $p$-modulus of general measures was initiated by Fuglede \cite{Fug57}. In Fuglede's definition, the $p$-modulus is defined for a family of measures in a metric measure space. The above definition of $p$-modulus is exactly the one of Fuglede for the family of measures $\{ \|T\| : T\in\Lambda\}$. Note that the $p$-modulus is an outer measure on the set of polyhedral $m$-chains $\poly m$. Given $\Lambda\subset \poly{m}^0$, the \textit{$p$-capacity $\capa_p(\Lambda)$} is defined by
$$\capa_p(\Lambda):=M_p(\Gamma),$$
where $\Gamma=\{ S\in\poly{m+1}: \bdry S=T \text{ for some } T\in\Lambda\}$. We define the \textit{$(q,p)$-capacity of a family $\Lambda \subset \poly{m}$} by
$$\capa_{q,p}(\Lambda):= \inf \left\{ \int f_1^q d\lm^n + \int f_2^p d\lm^n \right\},$$ where the infimum is taken over all non-negative Borel functions $f_1\in L^q(\R^n)$ and $f_2\in L^p(\R^n)$ satisfying $\int f_1 d\| R\| +\int f_2 d\| S\| \geq 1$ for every decomposition $R+\bdry S\in \Lambda$ with $R\in\poly m$ and $S\in \poly{m+1}$. Notice that $\capa_{q,p}(\Lambda)=0$ if and only if there exist $\Lambda_1\subset \poly{m}$ and $\Lambda_2 \subset \poly{m}^0$ with $M_q(\Lambda_1)=\capa_p(\Lambda_2)=0$ and such that if $T=R+\bdry S\in\Lambda$, then $R\in\Lambda_1$ or $\bdry S\in\Lambda_2$. This notion of $(q,p)$-capacity is adapted to the set of cochains $W_{q,p}(\poly{m})$ in the sense that $\|X\|_{q,p}=0$ if and only if $X(T)=0$ for every $T\in \poly{m}\setminus \Lambda$, where $\capa_{q,p}(\Lambda)=0$.

With the notion of modulus available, one defines weak versions of upper norms and upper gradients as follows. Given a cochain $X$ on $\poly m$ a Borel function $h:\R^n\to [0,\infty]$ is said to be a \textit{$q$-weak upper norm of $X$} if there exists $\Lambda\subset\poly{m}$ with $M_q(\Lambda)=0$ such that \eqref{eq:upper norm} holds for every $T\in \poly m \setminus \Lambda$. Similarly, a Borel function $g:\R^n \to [0,\infty]$ is said to be a \textit{$p$-weak upper gradient of $X$} if there exists $\Gamma\subset\poly{m+1}$ with $M_p(\Gamma)=0$ such that \eqref{eq:upper gradient} holds for every $S\in \poly{m+1} \setminus \Gamma$.

We will make frequent use of Fuglede's lemma \cite{Fug57} which, in our setting, reads as follows.

\begin{lemma}[Fuglede's lemma]   \label{lem:Fuglede}
Let $1\leq p<\infty$, and let $f$ be a Borel function. Moreover, let $(f_j)$ be a sequence of Borel functions converging to $f$ in $L^p(\R^n)$. Then there exist a subsequence $(f_{j_k})$ and $\Lambda\subset \poly{m}$ with $M_p(\Lambda)=0$ such that
$$\int_{\R^n} | f_{j_k}-f|\, d\| T\| \to 0$$
for every $T\in \poly{m} \setminus \Lambda$.
\end{lemma}

As a consequence, one obtains that if $1<q,p<\infty$ then $L^q(\R^n)$-bounded sequences of upper norms converge, up to a subsequence, to $q$-weak upper norms, and similarly, $L^p(\R^n)$-bounded sequences of upper gradients converge, up to a subsequence, to $p$-weak upper gradients (see \cite{RW13} for details). In particular, the infimum in the definition of the Sobolev norm of a cochain is attained by some $q$-weak upper norm and some $p$-weak upper gradient.

We end this section with the following useful observation proved in \cite[Proposition 4.17]{RW13}.

\begin{lemma}  \label{lem:modulus translates}
Let $0\leq m\leq n$ and $T\in\poly{m}$ with $T\neq 0$. Let $B\subset \R^n$ be a Borel set with $\lm^n(B)>0$. Then the set $\Lambda:= \{ {\varphi_x}_\#T: x\in B \}$ has $M_q(\Lambda)>0$ for every $q\geq 1$.
\end{lemma}

\subsection{Averages of cochains}

Let $0\leq m\leq n$ and $1\leq p,q\leq \infty$. Given an additive cochain $X\in \cW_{q,p}(\poly{m})$ and $r>0$, define an additive cochain $X_r:\poly m\to\overline\R$ by
$$X_r(T):= \fint_{B(0,r)} X({\varphi_x}_\#T) dx$$
for every $T \in\poly m$, where $\varphi_x:\R^n\to\R^n$ is the translation map $\varphi_x(y)=x+y$.  Here, $\fint_E$ denotes the integral average $\lm^n(E)^{-1}\int_E$ and $B(0,r)$ is the open ball of radius $r$ centered at $0$.
For additive $m$-cochains on integral or normal currents with $m\leq n-1$ the measurability and the local integrability of the function $x\mapsto X({\varphi_x}_\#T)$ was proved in \cite[Lemma 4.14(i)]{RW13}. We have the following analog for cochains on polyhedral chains.

\begin{lemma}\label{lem:mb-loc-int-transl}
 Let $0\leq m\leq n$ and $1\leq p,q\leq \infty$ and let $X\in \cW_{q,p}(\poly{m})$ be an additive cochain. Then for every $T\in\poly{m}$ the function $u:\R^n\to \overline{\R}$ given by $u(x):= X({\varphi_x}_\#T)$ is Lebesgue measurable and locally integrable.
\end{lemma}

The same result holds when $X$ is an additive cochain in $\cW_p(\poly{m}^0)$ and $T\in\poly{m}^0$ and we can thus define $X_r$ in this case as well. It is not difficult to show that  $X_r\in \cW_{q,p}(\poly{m})$ for every $r>0$ and furthermore $X_r\in \cW_{\infty,\infty}(\poly{m})$. 

\begin{proof}
If $m=n$ then it is straight-forward to check that $u$ is continuous. We may therefore assume that $m\leq n-1$. In this case, the proof of \cite[Lemma 4.14(i)]{RW13} shows that there exists a non-negative Borel measurable and locally integrable function $\bar{\nu}$ which is an upper gradient of $u$ with respect to polygonal curves, that is, such that $$|u(b) - u(a)|\leq \int_0^1\bar{\nu}\circ\gamma(t)\,|\dot{\gamma}(t)|\,dt$$ for all $a,b\in\R^n$ and every polygonal curve $\gamma$ connecting $a$ and $b$. Here, it is understood that the right-hand side must equal $\infty$ in case $u(a) = \infty$ or $u(b)=\infty$. (We remark here that the proof of \cite[Lemma 4.14(i)]{RW13} is stated only for $p,q<\infty$. However, the same arguments apply in the case that $q=\infty$ or $p=\infty$.)
It now follows from a well-known argument (see e.g.~page 28 in \cite{Hei01}) that $$|u(b) - u(a)| \leq C|b-a|(M\bar{\nu}(b) + M\bar{\nu}(a))$$ for all $a,b\in\R^n$, where $M\bar{\nu}$ is the maximal function of $\bar{\nu}$, and where $C$ is a constant only depending on $n$. This implies that $u$ is Lipschitz continuous on $\{M\bar{\nu} \leq k\}$ for every $k\in\N$ and, since $\cap\{M\bar{\nu}\geq k\}$ is negligible, it follows that $u$ is Lebesgue measurable. Finally, the upper norm inequality, Fubini's theorem and (in case that $q<\infty$) H\"older's inequality yield that $u$ is locally integrable.
\end{proof}

In the proof of Theorem~\ref{thm:smoothening} we will need the following two crucial facts from \cite{RW13} about the averages $X_r(T)$. The statements given in \cite{RW13} are slightly stronger and are proved in the generality of normal and integral currents in Lie groups equipped with a left-invariant Finsler metric. In the setting of $\R^n$ the results can be stated in a somewhat simpler form. We thus provide them here for the convenience of the reader.
The following is a restatement of \cite[Proposition 4.15(ii)]{RW13}.

\begin{proposition}\label{prop:averages-RW13}
 Let $0\leq m\leq n-1$ and $1\leq p<\infty$. Let furthermore $X\in\cW_p(\poly{m}^0)$ be an additive cochain with upper gradient $g$ in $L^p(\R^n)$. Then there exists $C=C(n,m,p)>0$ such that for every $T\in\poly{m}^0$, every $S\in\poly{m+1}$ with $\bdry S=T$, and every $r>0$, we have
 $$|X_r(T)| \leq C r^{-n/p} \mass(S) \|g\|_p.$$
\end{proposition}

In order to state the second proposition we
define the maximal growth of a polyhedral chain $T\in\poly{m}$ by
\begin{equation*}
 \Theta_m(T):= \sup\frac{\|T\|(B(x,r))}{r^m},
\end{equation*}
where the supremum is taken over all $x\in\R^n$ and all $r>0$.
Note that $\Theta_m(T)<\infty$ for every $T\in\poly{m}$. We now give a restatement of  \cite[Proposition 4.16]{RW13}. 

\begin{proposition}\label{prop:cont-averages-RW13}
Let $0\leq m\leq n-1$ and let $n-m+1<q<\infty$ and $n-m<p<\infty$.
\begin{enumerate} 
 \item Let $X\in\cW_p(\poly{m}^0)$ be an additive cochain with upper gradient $g$ in $L^p(\R^n)$. Then there exists a constant $D=D(p,m,n)>0$ such that for every $T\in\poly{m}^0$ and every $r>0$ we have 
 $$|X_r(T) - X(T)|\leq D\,  \Theta_m^{1/p}(T) r^{1+\frac{m-n}{p}} \mass(T)^{\frac{p-1}{p}} \| g\|_p.$$
 \item Let $X\in\cW_{q,p}(\poly{m})$ be an additive cochain with upper norm $h$ in $L^q(\R^n)$ and upper gradient $g$ in $L^p(\R^n)$. Then there exists $E=E(q,p,m,n)>0$ such that for every $T\in\poly{m}$ and every $r>0$ we have 
 \begin{equation*}
\begin{split}
 |X_r(T) - X(T)|\leq E & \left[  \Theta_m^{1/p}(T) r^{1+\frac{m-n}{p}} \mass(T)^{\frac{p-1}{p}} \| g\|_p \right. \\
 & + \left. \Theta_{m-1}^{1/q}(\bdry T) r^{1+\frac{m-1-n}{q}} \mass(\bdry T)^{\frac{q-1}{q}} \| h\|_q \right].
 \end{split}
 \end{equation*}
\end{enumerate}
\end{proposition}

The proof of this proposition is exactly as the proof of \cite[Proposition 4.16]{RW13} except that the reference to \cite[Lemma 4.14(i)]{RW13} therein should be replaced by a reference to Lemma~\ref{lem:mb-loc-int-transl} above.


\section{Continuity of averages}\label{sec:continuity}

The aim of this section is to prove Theorem~\ref{thm:smoothening} below, which provides the main continuity result needed in the proof of our generalizations of Wolfe's theorem. The second part of Theorem~\ref{thm:smoothening} was stated as Theorem~\ref{thm:cont-intro} in the introduction.

\begin{theorem}   \label{thm:smoothening}
Let $0\leq m\leq n$ and $1<p,q\leq \infty$. Then we have:
\begin{enumerate}
 \item if $X\in \cW_p(\poly{m}^0)$ is an additive cochain then
 $$| X_r(T)-X(T)|\to 0 \text{ as } r\to 0$$ for every $T\in\poly{m}^0\setminus \Lambda_1$, where $\Lambda_1\subset\poly{m}^0$ has $p$-capacity $0$. If $p>n-m$ then we may take $\Lambda_1=\emptyset$.
 \item if $X \in \cW_{q,p}(\poly{m})$ is an additive cochain then
$$| X_r(T)-X(T)|\to 0 \text{ as } r\to 0$$
for every $T\in \poly{m}\setminus \Lambda_2$ for some family $\Lambda_2 \subset \poly{m}$ of zero $\nu$-modulus, where $\nu=q$ if $p> n-m$ and
$$\nu=\min \{ q,pn/(n-p)\}$$
otherwise. If $q>n-m+1$ and $p> n - m$ then we may take $\Lambda_2 = \emptyset$.
\end{enumerate}
\end{theorem}

The proof of this theorem will be given in Section~\ref{subsec:proof-smoothening}. We first establish some preliminary results which will be used in its proof. 

\subsection{Auxiliary results}

We will need the following results in the proof of Theorem~\ref{thm:smoothening}.

 
\begin{lemma}    \label{lem:Riesz potential}
Let $0\leq m\leq n-1$ and let $T\in\poly{m}$. Let furthermore $u : \R^n \to [0,\infty]$ be a Borel function. For $x \in \R^n$ define 
$$S_x:={\psi_{x}}_\#([0,1]\times T)\in\poly{m+1},$$ 
where 
$$
\psi_x:[0,1]\times \R^n \to \R^n, \quad \psi_x(t,z)=z+tx. 
$$
Then for every $r>0$ we have
$$
\fint_{B(0,r)} \int_{\R^n} u(y) \, d\| S_x\|(y)dx \leq (n-1)^{-1} \lm^n(B(0,1))^{-1} \int_{\R^n} I_r(u)(y) \,d\|T\|(y),
$$
where $I_r$ is the truncated Riesz potential 
$$I_r(u)(y):= \int_{B(y,r)} \frac{u(x)}{|x-y|^{n-1}}\, dx.$$
\end{lemma}

\begin{proof}
By \cite[4.1.9]{Fed69}, we have
$$
\| S_x\| \leq r{\psi_{x}}_\#(\lm^1 \times \| T\|)
$$ 
for every $x\in B(0,r)$. 
Using polar coordinates and Fubini's theorem we calculate that
\begin{equation*}
\begin{split}
\int_{B(0,r)} \int_{\R^n} u(y) &\, d\| S_x\|(y)dx\\
&\leq r   \int_{B(0,r)} \int_0^1 \int_{\R^n} u(z+tx)\, d\| T\|(z) dt dx\\
&= r \int_0^r \tau^{n-1} \int_{S^{n-1}} \int_0^1 \int_{\R^n} u(z+t \tau \theta)\, d\|T \|(z) dt d\theta d\tau\\
&\leq r \int_0^r \tau^{n-2} \int_{S^{n-1}} \int_0^r \int_{\R^n} u(z+t \theta)\, d\|T \|(z) dt d\theta d\tau\\
&= \frac{r^n}{n-1} \int_{B(0,r)} \int_{\R^n} \frac{u(z+x)}{|x|^{n-1}}\, d\| T\|(z) dx\\
&= \frac{r^n}{n-1} \int_{\R^n} I_r(u)(y)\, d\| T\|(y),
\end{split}
\end{equation*}
which finishes the proof.
\end{proof}

The following notation will be useful in the sequel. For $T\in\poly{m}$ and $0<r<R$ set
 $$N(T,R,r)=N(T,R)\setminus \overline N(T,r),$$
 where $N(T,s)$ denotes the open $s$-neighborhood of the support of $T$. 
 
 \begin{proposition}   \label{prop:isoperimetric}
Let $0\leq m\leq n-1$ and $T\in\poly{m}^0$. Then there exist $r_0,A>0$ and $t\geq 1$ (depending on $T$) with the following property. For every $S\in\poly{m+1}$ with $\bdry S=T$ we have
 $$
 \|S \|\Big((B(y,r)\setminus \overline B(y,r/2))\cap N(T,r,t^{-1}r)\Big)\geq Ar \| T\|(B(y,2r))
 $$
 for every $0<r<r_0$ and every $y\in \spt(T)$.
 \end{proposition}

We postpone the proof of this proposition until Section~\ref{subsec:prop-isoperimetric} since it is quite different in spirit from the rest of the proofs in this section.
From Proposition~\ref{prop:isoperimetric} we deduce the following fact.

\begin{proposition}\label{prop:Riesz-maximal-fct-est}
 Let $0\leq m\leq n-1$ and $T\in\poly{m}^0$. Then there exist $C, r_0>0$ with the following property. For every  $S\in\poly{m+1}$ with $\bdry S = T$ we have
 \begin{equation*}
  \int_{R^n} I_r(h)(y)\,d\|T\|(y) \leq C \int_{N(T, r)} Mh(y) \,d\|S\|(y)
 \end{equation*}
 for every $0<r<r_0$ and every Borel function $h: \R^n\to[0, \infty]$. Here $Mh$ is the Hardy-Littlewood maximal function of $h$. 
\end{proposition}

\begin{proof}
Let $A$, $t$ and $r_0$ be as in Proposition \ref{prop:isoperimetric} for $T$, and fix $0<r<r_0$. For $j=0,1,\dots$ define
$$
A_j(y)=B(y,2^{-j}r)\setminus \overline B(y,2^{-(j+1)}r) \quad \text{and} \quad  N_j(T)=N(T,2^{-j}r, t^{-1}2^{-j}r).
$$ 
Set
$$H_j(y):=\int_{A_j(y)} \frac{h(x)}{|x-y|^{n-1}}dx,$$
and note that $I_r(h)(y)=\sum_j H_j(y)$. If $y\in\spt(T)$ then we have
$$
H_j(y)\leq D2^{-j}r Mh(w) \quad \text{ for every } w\in A_j(y)\cap N_j(T)=:Q_j(y), 
$$
where $D$ is a constant depending only on $n$.  
Consequently,
 $$H_j(y)\leq D2^{-j} r\left( \| S\|(Q_j(y))\right)^{-1} \int_{\R^n} Mh(w) \chi_{Q_j(y)}(w)\, d\| S\| (w).$$
 Integration and Fubini's theorem then yield
 \begin{equation*}
 \int_{\R^n} H_j(y)\, d\|T\|(y)
 \leq D \int_{\R^n} Mh(w) 2^{-j}r\chi_{N_j(T)}(w) \int_{\R^n} \frac{\chi_{B(w,2^{-j}r)}(y)}{\| S\|(Q_j(y))}\,d\|T\|(y) d\| S \|(w).
 \end{equation*}
 We apply Proposition \ref{prop:isoperimetric} to bound the right-hand-side from above by
 $$A^{-1} D \int_{\R^n} Mh(w) \chi_{N_j(T)}(w)\, d\| S\|(w).$$
 Summing over $j$ yields 
 \begin{equation*}
   \int_{R^n} I_r(h)(y)\,d\|T\|(y) \leq C\int_{\R^n} Mh(y) \chi_{N(T,r)}(y)\,d\| S\|(y),
 \end{equation*}
 with a constant $C$ depending on $T$ but not on $r$.
\end{proof}

\subsection{Proof of the main continuity result}\label{subsec:proof-smoothening}
We now turn to the proof of Theorem \ref{thm:smoothening}.

\begin{proof}[Proof of Theorem \ref{thm:smoothening}]
We first prove statement (i). For this, let $X\in \cW_{p}(\poly{m}^0)$ be an additive cochain with $p$-integrable upper gradient $g$. We may assume that $m\leq n-1$ because $\poly{n}^0=\{0\}$ and thus $X\equiv 0$ when $m=n$. Set
 $$\Lambda_1:= \left\{ T\in\poly{m}^0: \int_{\R^n} Mg \ d\| S\|=\infty \text{ for all } S\in\poly{m+1} \text{ such that } \bdry S=T \right\}.$$
If $p=\infty$ then it follows that $\Lambda_1=\emptyset$. If $p<\infty$ then $Mg\in L^p(\R^n)$ by the maximal function theorem and hence $\capa_p(\Lambda_1)=0$. 
Let $T\in\poly{m}^0\setminus\Lambda_1$. Then there exists $S\in\poly{m+1}$ with $\bdry S = T$ and such that 
\begin{equation}\label{eq:finite-int-Mg}
 \int_{\R^n} Mg \ d\| S\|<\infty.
 \end{equation}
For $x\in\R^n$ set $V_x:= {\psi_{x}}_\#([0,1]\times T)$, where $\psi_x$ is as in Lemma \ref{lem:Riesz potential}, and note that ${\varphi_{x}}_\#T-T= - \bdry V_x$.
By the upper gradient inequality and Lemma \ref{lem:Riesz potential} we have
\begin{equation*}
\fint_{B(0,r)} |X(\partial V_x)|\, dx \leq D\int_{\R^n} I_r(g)(y) \, d\| T\|(y)
\end{equation*}
for every $r>0$, where $D$ is a constant only depending on $n$. This together with Proposition~\ref{prop:Riesz-maximal-fct-est} yields
\begin{equation*}
\fint_{B(0,r)} |X(\partial V_x)|\, dx \leq C \int_{N(T, r)} Mg(y)\, d\|S\|(y)
\end{equation*}
for every $0<r<r_0$, where $C$ is a constant depending on $T$ and $n$ but not on $r$, and where $r_0>0$ is as in Proposition~\ref{prop:Riesz-maximal-fct-est}. 
Together with \eqref{eq:finite-int-Mg} this yields $|X(\partial V_x)|<\infty$ for almost every $x\in B(0,r)$. Note also that $|X({\varphi_{x}}_\#T)|<\infty$ for almost every $x\in B(0,r)$ by the upper gradient inequality and Lemma~\ref{lem:modulus translates}. The subadditivity property of $X$ thus implies that $|X(T)|<\infty$ and hence
 \begin{equation*}
|X_r(T)-X(T)| \leq \fint_{B(0,r)} |X(\partial V_x)|\, dx
\end{equation*}
by the additivity property of $X$. Consequently,
 \begin{equation*}
|X_r(T)-X(T)| \leq C \int_{N(T, r)} Mg(y)\, d\|S\|(y).
\end{equation*}
Given \eqref{eq:finite-int-Mg} we have
$$\int_{\R^n} Mg \ \chi_{N(T,r)} d\| S\| \to 0 \text{ as } r\to 0$$
 by absolute continuity of integrals, and thus we obtain $|X_r(T)-X(T)|\to 0$ as $r\to0$. 
 
 In order to complete the proof of statement (i) it remains to show that we may take $\Lambda_1=\emptyset$ when $n-m<p<\infty$. In this case it follows directly from Proposition~\ref{prop:cont-averages-RW13} that for every $T\in\poly{m}^0$ we have $|X_r(T) - X(T)|\to 0$ as $r\to 0$. (Alternatively, one can show that the family $\Lambda_1$ defined above is in fact empty. Indeed, for every $T\in\poly{m}^0$ and every $n-m<p<\infty$ we have $\capa_p(\{T\})>0$ by the proof of \cite[Proposition~4.17]{RW13} and hence $\Lambda_1=\emptyset$. We remark that \cite[Proposition~4.17]{RW13} is stated in the setting of integral or normal currents but the same arguments apply in the setting of polyhedral chains.)
 This completes the proof of statement (i).
 
 We turn to statement (ii). Let $X\in \cW_{q,p}(\poly{m})$ be an additive cochain with $q$-integrable upper norm $h$ and $p$-integrable upper gradient $g$. If $m=n$ then the function $x\mapsto X({\varphi_x}_\#T)$ is continuous and thus the statement holds trivially. We may therefore assume from now on that $m\leq n-1$. Given $T\in\poly{m}$ and $x\in\R^n$ we can write 
$$
{\varphi_{x}}_\#T-T= {\psi_{x}}_\#([0,1]\times \partial T)-\partial {\psi_{x}}_\#([0,1]\times T)=: U_x - \partial V_x, 
$$
where $\psi_x$ is as in Lemma \ref{lem:Riesz potential}. Note that if $m=0$ then $U_x=0$ by definition. 
We now distinguish two cases. First assume that $1<p\leq n-m$ or $p=\infty$. Set $\Lambda_2:= \Lambda_2'\cup\Lambda_2''$, where
 $$\Lambda_2':=\left\{ T\in\poly{m}: \int_{\R^n} Mh \ d\| T\|=\infty \right\}$$
 and
 $$\Lambda_2'':=\left\{ T\in\poly{m}: \int_{\R^n} I_1(g) \ d\| T\|=\infty \right\}.$$
 Here, $I_1(g)$ denotes the truncated Riesz potential with $r=1$ defined in Lemma~\ref{lem:Riesz potential}. 
Note that if $p=\infty$ then $\Lambda_2''=\emptyset$, and if $q=\infty$ then $\Lambda_2'=\emptyset$. In particular, it follows that $\Lambda_2=\emptyset$ in the case that $q=p=\infty$. If $q<\infty$ then $M_q(\Lambda_2')=0$ since $Mh\in L^q(\R^n)$ by the maximal function theorem. Also, if $1< p\leq n-m$ then $M_{pn/(n-p)}(\Lambda_2'')=0$ since $I$ maps $L^p(\R^n)$ to $L^{np/(n-p)}(\R^n)$, cf. \cite[page 20]{Hei01}. Since $M_{p_0}(\Lambda)=0$ implies $M_{q_0}(\Lambda)=0$ for $q_0<p_0$, we have $M_\nu (\Lambda_2)=0$.
Now, let $T\in\poly{m}\setminus\Lambda_2$. Clearly,
\begin{equation}\label{eq:fin-int-caseii}
\int_{\R^n} (Mh(y) + I_1(g)(y)) \,d\|T\|(y)<\infty.
\end{equation}
From the upper gradient inequality and Lemma~\ref{lem:Riesz potential} we infer
\begin{equation}\label{eq:est-fint-Vx}
 \fint_{B(0,r)} |X(\partial V_x)|\, dx \leq D\int_{\R^n} I_r(g)(y) \, d\| T\|(y)
\end{equation}
for every $r>0$, where $D$ is a constant only depending on $n$. Analogously, from the upper norm inequality and  Lemma~\ref{lem:Riesz potential} we obtain
 \begin{equation*}
 \fint_{B(0,r)} |X(U_x)|\, dx \leq D \int_{\R^n} I_r(h)(y) \,d\| \bdry T\|(y)
\end{equation*}
for every $r>0$. Proposition~\ref{prop:Riesz-maximal-fct-est} thus implies
\begin{equation}\label{eq:est-fint-Ux}
  \fint_{B(0,r)} |X(U_x)|\, dx \leq C\int_{\R^n} Mh(y) \chi_{N(\bdry T,r)}(y)  \,d\| T\|(y)
\end{equation}
for every $0<r<r_0$, where $C$ is a constant depending on $\bdry T$ and $n$, but not on $r$. It thus follows together with \eqref{eq:fin-int-caseii} that $|X(U_x)|<\infty$ and $|X(\partial V_x)|<\infty$ for almost every $x\in B(0,r)$. Moreover, $|X({\varphi_{x}}_\#T)|<\infty$ for almost every $x\in B(0,r)$ by the upper norm inequality and Lemma~\ref{lem:modulus translates}. This together with the subadditivity property of $X$ shows that $|X(T)|<\infty$. Now, the additivity property of $X$ yields 
\begin{equation*}
 |X_r(T)-X(T)|\leq \fint_{B(0,r)} |X(U_x)|\, dx + \fint_{B(0,r)} |X(\partial V_x)|\, dx
\end{equation*}
and therefore, in view of \eqref {eq:est-fint-Vx} and \eqref{eq:est-fint-Ux},
\begin{equation}\label{eq:cont-p-small}
|X_r(T)-X(T)|\leq C\int_{\R^n} \left(Mh(y) \chi_{N(\bdry T,r)}(y) + I_r(g)(y) \right) \,d\| T\|(y)
\end{equation}
for every $0<r<r_0$, where $C$ is a constant depending on $\bdry T$ and $n$, but not on $r$. Given \eqref{eq:fin-int-caseii}, it follows from absolute continuity of the integral and monotone convergence that the right side in \eqref{eq:cont-p-small} converges to $0$ as $r\to 0$ and thus $|X_r(T) - X(T)| \to0$ as $r\to0$. This proves statement (ii) when $1<p\leq n-m$ or $p=\infty$.

Now assume that $n-m<p<\infty$ and set $$\Lambda_2:= \left\{ T\in\poly{m}: \int_{\R^n} Mh \ d\| T\|=\infty \right\}.$$ Observe that if $q=\infty$ then $\Lambda_2=\emptyset$. Let $T\in\poly{m}\setminus\Lambda_2$. Firstly, as above, the upper norm inequality together with Lemma~\ref{lem:Riesz potential} and Proposition~\ref{prop:Riesz-maximal-fct-est} yields
\begin{equation}\label{eq:fint-Ux-caseii}
 \fint_{B(0,r)} |X(U_x)|\, dx \leq C\int_{\R^n} Mh(y) \chi_{N(\bdry T,r)}(y)  \,d\| T\|(y)
\end{equation}
for every $0<r<r_0$, where $C$ is a constant depending on $\bdry T$ and $n$, but not on $r$. In particular, $|X(U_x)|<\infty$ for almost every $x\in B(0,r)$. Moreover, as above, the integral on the right hand side of \eqref{eq:fint-Ux-caseii} converges to $0$ as $r\to 0$ and hence $$\fint_{B(0,r)} |X(U_x)|\, dx\to 0$$ as $r\to 0$.
Secondly, Propositions~\ref{prop:averages-RW13} and \ref{prop:cont-averages-RW13} yield, with $s=r^\alpha$ for $0<\alpha<\frac{p}{n}$, that for suitable constants $C,D$, and $C'$ depending only on $p,n,m$ we have 
\begin{equation*}
 \begin{split}
  |X(\partial V_x)| &\leq |X(\partial V_x) - X_s(\partial V_x)| + |X_s(\partial V_x)|\\
  &\leq C\Theta_m^{\frac{1}{p}}(\partial V_x) s^{1+\frac{m-n}{p}} \mass(\partial V_x)^{\frac{p-1}{p}} \|g\|_p + Ds^{-\frac{n}{p}}\mass(V_x)\|g\|_p\\
  &\leq C'[\Theta_m(T)+\Theta_{m-1}(\partial T)]^{\frac{1}{p}} r^{\alpha\left(1+\frac{m-n}{p}\right)}( \mass(T)+ r\mass(\partial T))^{\frac{p-1}{p}}\|g\|_p\\
  &\quad + Dr^{1-\alpha\frac{n}{p}}\mass(T)\|g\|_p
 \end{split}
\end{equation*}
for every $x\in B(0,r)$, where we have used $\mass(V_x)\leq |x|\,\mass(T)$ as well as $\mass(\partial V_x)\leq 2\mass(T) + |x|\mass(\partial T)$ and $\Theta_m(\partial V_x)\leq 2\Theta_m(T) + 2^{m-1}\Theta_{m-1}(\partial T)$. In particular, $|X(\partial V_x)|<\infty$ for every $x\in B(0,r)$ and, moreover, $$\fint_{B(0,r)} |X(\partial V_x)|\, dx\to 0$$ as $r\to0$. Thus, we see the same way as above that $|X(T)|<\infty$ and hence, with the additivity property of $X$, that
\begin{equation*}
 |X_r(T)-X(T)|\leq \fint_{B(0,r)} |X(U_x)|\, dx + \fint_{B(0,r)} |X(\partial V_x)|\, dx \to 0
\end{equation*}
as $r\to 0$. This shows that $|X_r(T)-X(T)|\to 0$ as $r \to 0$. 
 
 It remains to show that we may take $\Lambda_2=\emptyset$ in the case that $n-m+1<q<\infty$ and $n-m<p<\infty$. In this case, Proposition~\ref{prop:cont-averages-RW13} in fact yields that for every $T\in\poly{m}$ we have  $|X_r(T)-X(T)|\to 0$ as $r \to 0$. This concludes the proof of statement (ii) and thus of the theorem.
\end{proof}

\subsection{Proof of Proposition~\ref{prop:isoperimetric}}\label{subsec:prop-isoperimetric}
In order to prove Proposition~\ref{prop:isoperimetric} we will need the following result.

\begin{lemma}\label{lem:LNR-simplicial-complex}
 Let $X\subset\R^n$ be a finite simplicial complex. Then $X$ is a local Lipschitz neighborhood retract.
\end{lemma}

\begin{proof}
Let $X\subset\R^n$ be a finite simplicial complex. By  \cite[Theorem 1.2]{Alm62}, it is enough to show that there exist $C>0$ and $\varepsilon>0$ such that every Lipschitz map $f: S^r\to X$ with image in an $\varepsilon$-ball admits a Lipschitz extension $\bar{f}: B^{r+1}\to X$ whose Lipschitz constant is bounded by $C$ times the Lipschitz constant of $f$.

We start with the following auxiliary construction.
 Let $\Sigma\subset\R^n$ be a $k$-simplex and $F\subset\Sigma$ an $\ell$-face, possibly $F=\Sigma$. We can write $\Sigma = [v_0,\dots, v_k]$ and $F = [v_0,\dots, v_\ell]$. Set $$\Sigma_F := \left\{\sum_{i=0}^k t_i v_i \in\Sigma: \sum_{i=0}^\ell t_i \geq 2^{-1}\right\}$$ and define $$\psi_F^\Sigma: [0,1]\times\Sigma_F \to \Sigma_F$$ by $$\psi_F^\Sigma(s, x) = sx + (1-s)\left(\sum_{i=0}^\ell t_i\right)^{-1} \sum_{i=0}^\ell t_i v_i,$$
 where $x = \sum_{i=0}^k t_iv_i$. It is clear that the following properties hold:
 \begin{enumerate}
  \item $\psi_F^\Sigma(1, x) = x$ and  $\psi_F^\Sigma(0,x)\in F$ for all $x\in\Sigma_F$;
  \item $\psi_F^\Sigma(s,x) = x$ for all $x\in F$ and $s\in[0,1]$;
  \item $\psi_F^\Sigma$ is Lipschitz with constant depending on $\Sigma$;
  \item If $G\subset\Sigma$ is a face with $F\cap G \not=\emptyset$ then $\Sigma_F\cap G = G_{G\cap F}$ and $$\psi_F^\Sigma(s,x) = \psi_{F\cap G}^G(s,x)$$ for all $x\in G_{G\cap F}$ and all $s\in[0,1]$. 
 \end{enumerate}
 
Next, fix a simplex $\Sigma_0$ in $X$ and define a subset $A\subset X$ by
 $$A:= \bigcup\left\{\Sigma_{\Sigma\cap\Sigma_0}: \text{ $\Sigma$ simplex in $X$}\right\}.$$
 There exists $\varepsilon_0>0$ depending only on $X$ such that $$N(\Sigma_0, \varepsilon_0)\cap X \subset A,$$ where $N(\Sigma_0, \varepsilon_0)$ denotes the $\varepsilon_0$-neighborhood of $\Sigma_0$ in $\R^n$. Define a map $$\varphi: [0,1]\times A \to A$$ by $\varphi(s,x) = \psi_{\Sigma\cap\Sigma_0}^\Sigma(s,x)$ for $x\in \Sigma_{\Sigma\cap\Sigma_0}$. By property (iv) above, this map is well-defined, that is, independent of the choice of $\Sigma$. From properties (i) -- (iii) it follows that $\varphi(1,x) = x$ for all $x\in A$, $\varphi(s,x) = x$ for all $x\in \Sigma_0$ and all $s\in[0,1]$, and $\varphi(0,x)\in \Sigma_0$ for all $x\in A$. Moreover, $\varphi$ is ``piecewise Lipschitz".
 
Finally, let $f: S^r\to X$ be a Lipschitz map with image in an $\varepsilon$-ball centered at some $x_0\in X$, where $0<\varepsilon<\varepsilon_0$ is so small that every $\varepsilon$-ball in $X$ is quasi-convex. We will show that $f$ admits a Lipschitz extension $\bar{f}: B^{r+1}\to X$ whose Lipschitz constant is bounded by $C$ times the Lipschitz constant of $f$, where $C$ only depends on $X$. If $r=0$ then this follows immediately from the quasi-convexity of $\varepsilon$-balls. Let therefore $r\geq 1$. Let $\Sigma_0\subset X$ be a simplex such that $x_0\in\Sigma_0$. With the definition of $A$ above, we clearly have $f(S^r)\subset A$. Using $\varphi$ we easily construct a Lipschitz extension $\bar{f}: B^{r+1}\to X$ of $f$ with Lipschitz constant $\Lipconst(\bar{f})\leq C\Lipconst(f)$, where $C$ only depends on $X$. Indeed, let $\varrho:[0,1]\times \Sigma_0\to\Sigma_0$ be a Lipschitz map which contracts $\Sigma_0$ to a point. Define 
 \begin{equation*}
  \bar{f}(sz):=\left\{ \begin{array}{ll}
   \varphi(2s-1, f(z)) & s\in[1/2, 1]\\
   \varrho(2s, \varphi(0, f(z))) & s\in [0,1/2)
  \end{array}\right.
 \end{equation*}
 whenever $z\in S^r$ and $s\in [0,1]$. Then $\bar{f}$ extends $f$ and is Lipschitz with a constant only depending on the ``piecewise" Lipschitz constant of $\varphi$ and the Lipschitz constant of $\varrho$. This proves the claim and thus shows, by \cite[Theorem 1.2]{Alm62}, that $X$ is a local Lipschitz neighborhood retract.
 \end{proof}

We are finally ready to prove Proposition~\ref{prop:isoperimetric}.

\begin{proof}[Proof of Proposition~\ref{prop:isoperimetric}]
 Since $T$ is a polyhedral cycle it follows from Lemma~\ref{lem:LNR-simplicial-complex} that $\spt T$ is a local Lipschitz neighborhood retract. There thus exist $r_1>0$, $\lambda\geq 1$ and a $\lambda$-Lipschitz retraction $\varrho: N(T, r_1) \to\spt T$. Let $y\in\spt T$ and let $u$ be the distance function to the point $y$. By \cite[4.2.1 and 4.3.2]{Fed69}, almost every $s\in(0,r_1/\lambda)$ is such that the slice $$\langle S, u, s\rangle = \bdry(S\rstr\{u\leq s\}) - (\bdry S)\rstr\{u\leq s\}$$ is a normal $m$-current supported in $\{x: u(x)=s\}$. Clearly, $\overline{B}(y,s)\subset N(T, r_1)$ and hence $\langle S, u, s\rangle$ is supported in $N(T,r_1)$. We claim that
 \begin{equation}\label{eq:proj-slice}
  \varrho_{\#}\langle S, u, s\rangle = - T\rstr B(y,s).
 \end{equation}
 In order to see this, set $V:= \varrho_{\#}\langle S, u, s\rangle + T\rstr B(y,s)$ and note first that $$\partial V = \varrho_{\#}(\partial \langle S, u, s\rangle) + \partial(T\rstr B(y,s))=-\varrho_{\#}\langle T, u, s\rangle +  \partial(T\rstr B(y,s)) = 0,$$ hence $V$ is a cycle. Note that $V$ is supported in $\spt T \cap \overline{B}(y, s/\lambda)$. Let $W$ be an $(m+1)$-dimensional normal current with $\partial W =  V$. After possibly projecting $W$ onto $\overline{B}(y,s/\lambda)$ we may assume that $W$ is supported in the ball $\overline{B}(y,s/\lambda)$ and thus in $N(T,r_1)$.  It follows that $\varrho_{\#}W$ satisfies $\partial \varrho_{\#}W = \varrho_{\#}V = V$. Since $\varrho_{\#}W$ is supported in an $m$-dimensional simplicial complex it follows that $\varrho_{\#}W = 0$ and hence that $V = \partial (\varrho_{\#}W) = 0$. This proves \eqref{eq:proj-slice}. 
 
 Now set $t:= 400\lambda$ and $r_0:= r_1/\lambda$. If $0<r<r_0$ and $s\in(r/2, r)$ then it follows that
 \begin{equation*}
  \|\varrho_{\#}(\langle S, u, s\rangle\rstr N(T,t^{-1}r))\|(B(y,99s/100)) = 0
 \end{equation*}
 and thus, with \eqref{eq:proj-slice}, that 
 \begin{equation*}
 \begin{split}
 \|T\|(B(y,99s/100))&= \|\varrho_{\#}(\langle S, u, s\rangle\rstr N(T, t^{-1}r)^c)\|(B(y,99s/100))\\
  &\leq \lambda^m \|\langle S, u, s\rangle\|(N(T, t^{-1}r)^c).
 \end{split}
 \end{equation*}
 Now, integration and the slicing inequality (see \cite[4.2.1]{Fed69} or \cite[Theorem 6.2]{Lan11}) yield
 \begin{equation*}
  \begin{split}
  \frac{r}{2} \|T\|(B(y,99r/200))&\leq \lambda^m \int_{r/2}^r  \|\langle S, u, s\rangle\|(N(T, t^{-1}r)^c) ds\\
   & \leq \lambda^m\|S\|\left((B(y,r)\backslash \overline{B}(y,r/2))\cap N(T, t^{-1}r)^c\right).
  \end{split}
 \end{equation*}
Since $T$ is polyhedral there furthermore exists $D\geq 1$ such that $$\|T\|(B(y,2r))\leq D\|T\|(B(y,99r/200))$$ for every $y\in\spt T$ and every $r>0$. We thus conclude that
 $$Ar \|T\|(B(y,2r)) \leq \|S\|\left((B(y,r)\backslash \overline{B}(y,r/2))\cap N(T, t^{-1}r)^c\right) $$
 for every $y\in\spt T$ and all $0<r<r_0$, where $A:= (2D\lambda^m)^{-1}$. Since $$(B(y,r)\backslash \overline{B}(y,r/2))\cap N(T, t^{-1}r)^c = (B(y,r)\backslash \overline{B}(y,r/2))\cap N(T, r, t^{-1}r)$$ this concludes the proof.
\end{proof}


\section{From Sobolev differential forms to Sobolev cochains}
\label{sec:form}
The aim of this section is to construct a linear map $$\Psi_m: W_d^{q,p}(\R^n,\wedge^m) \to W_{q,p}(\poly{m})$$ for all  $1\leq m\leq n$ whenever  $1\leq q,p< \infty$ or $q=p=\infty$, and to show that $\Psi_m$ is isometric, i.e.\ it preserves norms. 
This construction already appeared in \cite{RW13}. The authors did not show, however, that the map is isometric. In fact, as mentioned earlier, a different but equivalent norm on $W_d^{q,p}(\R^n,\bigwedge^m)$ was used in \cite{RW13}.

\subsection{Construction of $\Psi_m$ when $p,q<\infty$}   \label{subsec:construction cochain}

Let $1\leq p,q<\infty$ and $\omega\in W_d^{q,p}(\R^n,\bigwedge^m)$. There is a sequence of smooth compactly supported $m$-forms $\omega_j$ converging to $\omega$ in $W_d^{q,p}(\R^n,\bigwedge^m)$.
We may assume, of course, that the coefficients $\omega(\cdot,\alpha)$ of $\omega$ and $d\omega(\cdot,\beta)$ of $d\omega$ are Borel functions. Therefore, by Fuglede's lemma (Lemma \ref{lem:Fuglede}), there is a subsequence of $(\omega_j)$, also denoted $(\omega_j)$, such that for every $\alpha\in\Lambda(m,n)$ and every $\beta\in\Lambda(m+1,n)$,

\begin{equation}\label{eq:Lq-conv-omegas}
\int_{\R^n}|\omega_j(\cdot,\alpha)-\omega(\cdot,\alpha)| \, d\| T\| \to 0
\end{equation}
for every $T\in\poly{m} \setminus \Lambda$, where $M_q(\Lambda)=0$, and

\begin{equation}\label{eq:Lp-conv-domegas}
 \int_{\R^n}|d\omega_j(\cdot,\beta)-d\omega(\cdot,\beta)|\, d\| S\| \to 0
\end{equation}
for every $S\in\poly{m+1} \setminus \Gamma$, where $M_p(\Gamma)=0$. 

For $T\in\poly{m}$ and $j\in\N$, define
$$X^{\omega_j}(T):=\int_T \omega_j.$$

Stokes theorem implies that $X^{\omega_j}(\partial S)=X^{d\omega_j}(S)$ and hence that $X^{d\omega_j} = dX^{\omega_j}$, where $dX^{d\omega_j}$ is the coboundary of $X^{\omega_j}$.
%
It is immediate that $X^{\omega_j}$ is an additive cochain. Define $X^\omega:\poly{m}\to\overline\R$ by $X^\omega(T):=\lim_{j\to\infty} X^{\omega_j}(T)$ when the limit exists, and $\infty$ otherwise. Note that, by Fuglede's lemma, the limit exists for $M_q$-almost every $T\in\poly{m}$ and, moreover, a different choice of $(\omega_j)$ satisfying \eqref{eq:Lq-conv-omegas} yields a cochain which is $M_q$-almost everywhere equal to $X^\omega$. It is clear that $X^\omega$ is additive and an element of $\cW_{q,p}(\poly{m})$. In fact, we have the following proposition.

\begin{proposition}\label{prop:Psi-m-isometric}
Let $1\leq m\leq n$ and $1\leq p,q<\infty$. Then the map $$\Psi_m: W_d^{q,p}(\R^n,\wedge^m) \to W_{q,p}(\poly{m})$$ given by $\Psi_m(\omega) = X^\omega$ is linear and isometric. Moreover, if $m<n$, we have
 \begin{equation}\label{eq:psi-commutes-d}
 \Psi_{m+1}\circ d = d\circ \Psi_m,
 \end{equation}
that is, for every $\omega\in W_d^{q,p}(\R^n,\bigwedge^m)$ and every $s\geq 1$ we have $X^{d\omega} = dX^\omega$ as elements of $W_{p,s}(\poly{m+1})$.
 \end{proposition}

\begin{proof}
Let $\omega\in W_d^{q,p}(\R^n,\bigwedge^m)$. Observe that the functions $\|\omega(x)\|$ and $\|d\omega(x)\|$ are in $L^q(\R^n)$ and $L^p(\R^n)$, respectively, where $\|\cdot \|$ denotes comass. It is not difficult to check that $\|\omega(x)\|$ is a $q$-weak upper norm of $X^\omega$ and $\|d\omega(x)\|$ is a $p$-weak upper gradient of $X^\omega$. Indeed, let $(\omega_j)$ be a sequence of smooth compactly supported $m$-forms converging to $\omega$ in $W_d^{q,p}(\R^n,\bigwedge^m)$ and satisfying \eqref{eq:Lq-conv-omegas} and \eqref{eq:Lp-conv-domegas}. Let $T\in\poly{m}\setminus \Lambda$, where $\Lambda$ is as above. Write $T=\sum a_i T_i$ as in Section \ref{subsec:polyhedral chains} and let $\tau_i$ be the simple unit $m$-vector orienting $T_i$. It follows that 
\begin{equation*}
 \begin{split}
  |X^{\omega_j}(T)| & \leq \sum_i |a_i| \int_{T_i} |\langle \omega_j, \tau_i\rangle| \,d\haus^m\\
  &\leq \sum_i|a_i|\int_{T_i} |\langle \omega, \tau_i\rangle|\,d\haus^m + \sum_i|a_i|\int_{T_i} |\langle \omega_j - \omega, \tau_i\rangle|\,d\haus^m.
 \end{split}
\end{equation*}
By \eqref{eq:Lq-conv-omegas}, each term in the second sum in the last line converges to $0$ as $j\to\infty$ and hence we obtain that $$|X^\omega(T)| \leq \sum_i|a_i|\int_{T_i} |\langle \omega, \tau_i\rangle|\,d\haus^m\leq \int_{\R^n} \|\omega(x)\|\,d\|T\|(x).$$ This shows that $\|\omega(x)\|$ is a $q$-weak upper norm of $X^\omega$. One shows analogously that $\|d\omega(x)\|$ is a $p$-weak upper gradient of $X^\omega$. Consequently, we have that 
\begin{equation}\label{eq:ineq-omega-Xomega}
\|X^\omega \|_{q,p} \leq \|\omega\|_{q,p}.
\end{equation}
We claim that equality holds in \eqref{eq:ineq-omega-Xomega}. For this let $h\in L^q(\R^n)$ be an upper norm of $X^\omega$. Fix $\tau = \xi_1\wedge\dots\wedge \xi_m$, where $\xi_1, \dots, \xi_m\in\R^n$ are pairwise orthonormal, and define for each $x\in\R^n$ a map  $\varphi_{\tau, x}:\R^m\to\R^n$ by $\varphi_{\tau, x}(z):= x+ \sum_{i=1}^m z_i \xi_i$ for $z=(z_1,\dots, z_m)\in\R^m$. We will show that for almost every $x\in\R^n$ and every polyhedron $\Delta$, we have 
\begin{equation}\label{eq:rep-Xomega-by-omega}
X^\omega({\varphi_{\tau, x}}_\#\Lbrack \chi_{\Delta}\Rbrack) = \int_{\Delta} \langle \omega\circ\varphi_{\tau, x}(z), \tau\rangle \,dz,
\end{equation}
where $\dblebracket{\chi_{\Delta}}$ denotes the polyhedral $m$-chain in $\R^m$ induced by the simple function $\chi_{\Delta}$.
By Lemma \ref{lem:modulus translates}, $X^{\omega_j}({\varphi_{\tau, x}}_\#\Lbrack \chi_{\Delta}\Rbrack)$ converges as $j\to\infty$ for almost every $x\in\R^n$. Let $V_\tau\subset\R^n$ be the span of the vectors $\xi_1,\dots, \xi_m$ and let $V_\tau^\perp\subset\R^n$ denote the orthogonal complement. Since $\omega_j(\cdot, \alpha)$ converges in $L^q(\R^n)$ to $\omega(\cdot, \alpha)$ for every $\alpha$ it follows that $\langle \omega_j, \tau\rangle$ converges in $L^q(\R^n)$ to $\langle \omega, \tau\rangle$. From this we obtain that there exists a subsequence $(\omega_{j_k})$ such that $\langle \omega_{j_k}\circ\varphi_{\tau, y}, \tau\rangle$ converges in $L^q(\R^m)$ to $\langle \omega\circ\varphi_{\tau, y}, \tau\rangle$ for almost every $y\in V_\tau^\perp$. In particular, for all such $y$ and every $z_0\in \R^m$ we have 
 $$X^{\omega_{j_k}}({\varphi_{\tau, y}}_\#\Lbrack \chi_{z_0+\Delta}\Rbrack) = \int_{z_0+\Delta}\langle \omega_{j_k}\circ\varphi_{\tau, y}(z), \tau\rangle\,dz \longrightarrow  \int_{z_0+\Delta} \langle \omega\circ\varphi_{\tau, y}(z), \tau\rangle \,dz$$ as $k\to\infty$. This together with the above implies \eqref{eq:rep-Xomega-by-omega}.
Now, since $h\in L^q(\R^n)$ it follows that $h\circ\varphi_{\tau, y}$ is in $L^q(\R^m)$ for almost every $y\in V_\tau^\perp$ and thus from Lebesgue differentiation theorem that for almost every $x\in\R^n$ 
\begin{equation*}
 \frac{1}{r^m}\int_{[0,r]^m}h\circ\varphi_{\tau, x}(z)\, dz\longrightarrow h(x)
\end{equation*}
as $r\to 0^+$. Analogously, we have that for almost every $x\in\R^n$
\begin{equation*}
  \frac{1}{r^m}\int_{[0,r]^m}\langle \omega\circ\varphi_{\tau, x}(z), \tau\rangle\, dz \longrightarrow \langle \omega(x), \tau\rangle
\end{equation*}
as $r\to0^+$. Fix a sequence $(r_k)$ with  $r_k\to 0^+$. Then from the above combined with \eqref{eq:rep-Xomega-by-omega} and the upper norm inequality we obtain that for almost every $x\in\R^n$ we have
\begin{equation*}
 \begin{split}
  |\langle \omega(x), \tau\rangle| & = \left|\lim_{k\to\infty} \frac{1}{r_k^m}\int_{[0,r_k]^m} \langle \omega\circ\varphi_{\tau, x}(z), \tau\rangle\,dz\right|\\
   &= \lim_{k\to\infty} \frac{1}{r_k^m}\left|X^\omega({\varphi_{\tau, x}}_\#\Lbrack \chi_{[0,r_k]^m}\Rbrack)\right|\\
   &\leq \liminf_{k\to\infty} \frac{1}{r_k^m} \int_{[0,r_k]^m}h\circ\varphi_{\tau, x}(z)\,dz\\
   &= h(x).
 \end{split}
\end{equation*}
Finally, let $\{\tau_k\}$ be a countable dense set of simple unit $m$-vectors. It follows from the above that for almost every $x\in\R^n$ we have $|\langle \omega(x), \tau_k\rangle| \leq h(x)$ for all $k$ and this shows that $\|\omega(x)\|\leq h(x)$ for almost every $x\in\R^n$. An analogous argument shows that if $g\in L^p(\R^n)$ is an upper gradient of $X^\omega$ then $\|d\omega(x)\|\leq g(x)$ for almost every $x\in\R^n$. This proves that $\|\omega\|_{q,p}\leq \|X^\omega\|_{q,p}$ and hence
equality holds in \eqref{eq:ineq-omega-Xomega} for all $\omega\in W_d^{q,p}(\R^n,\bigwedge^m)$. 

It remains to prove \eqref{eq:psi-commutes-d}. By Stokes' theorem, we have for every $j$ and every $S\in\poly{m+1}$,
$$X^{d\omega_j}(S)=dX^{\omega_j}(S).$$
By Fuglede's lemma, $X^{d\omega_j}(S) \to X^{d\omega}(S)$ for $M_p$-almost every $S\in \poly{m+1}$. Moreover, by definition, $X^{\omega}(\bdry S)=\lim X^{\omega_j}(\bdry S)$ when the limit exists. In particular, $dX^{\omega}(S)=\lim dX^{\omega_j}(S)$ for $M_p$-almost every $S\in\poly{m+1}$. This proves that for $M_p$-almost every $S\in\poly{m+1}$ we have
$$dX^\omega(S)=X^{d\omega}(S)$$
and hence that \eqref{eq:psi-commutes-d} holds.
\end{proof}

The following consequence of \eqref{eq:rep-Xomega-by-omega} will be used later (recall the notation $\polypar{m}$ from Section 
\ref{subsec:polyhedral chains}).

\begin{remark}   \label{rk:cochain from form}
For every $T\in\polypar{m}$ and almost every $x\in\R^n$, we have 
$$X^\omega({\varphi_x}_\#T)=\int_{{\varphi_x}_\#T} \omega.$$
Similarly, for every $S\in\polypar{m+1}$ and almost every $x\in\R^n$, 
$$X^\omega({\varphi_x}_\#\bdry S)=\int_{{\varphi_x}_\#S} d\omega.$$
\end{remark}

\subsection{Construction of $\Psi_m$ when $q=p=\infty$}  \label{sec:construction Psi case infinity}

Let $\omega \in W_d^{\infty,\infty}(\R^n,\bigwedge^m)$. We first assume that $\omega$ is compactly supported. Let $n-m+1< s <\infty$. Then clearly, $\omega$ is in $W_d^{s,s}(\R^n,\bigwedge^m)$. Let $X^\omega$ be the additive cochain in $\cW_{s,s}(\poly{m})$ induced by $\omega$, constructed as in Section \ref{subsec:construction cochain}. It follows that $\| \omega(\cdot)\|$ is an $s$-weak upper norm and that $\| d\omega(\cdot)\|$ is an $s$-weak upper gradient of $X^\omega$. It follows from Lemma \ref{lem:modulus translates} that for every $r>0$ the constant functions $\| \omega\|_\infty$ and $\| d\omega\|_\infty$ are an upper norm and an upper gradient of the averaged cochain
$$X_r^{\omega}(T)=\fint_{B(0,r)} X^{\omega}({\varphi_x}_\# T)dx,$$
respectively. The same is true for the cochain $X^\omega$ since $X_r^\omega(T)$ converges to $X^\omega(T)$ as $r\to 0$ for every $T\in\poly{m}$ by Theorem \ref{thm:smoothening}. In particular, we obtain that $X^\omega\in \cW_{\infty,\infty}(\poly{m})$ and $\| X^\omega\|_{\infty,\infty} \leq \| \omega\|_{\infty,\infty}$. The same proof as in Section \ref{subsec:construction cochain} shows that, in fact, $\| X^\omega\|_{\infty,\infty} = \| \omega\|_{\infty,\infty}$.

We now turn to the case where $\omega\in W_d^{\infty,\infty}(\R^n,\bigwedge^m)$ is not assumed to be compactly supported.  For $k\in \N$, let $\varphi_k$ be a smooth compactly supported function with the following properties: $\varphi_k$ takes values between $0$ and $1$, equals $1$ on $B(0,k)$, and $|\nabla \varphi_k|$ is bounded by $1/k$. Then the form $\omega^k:= \varphi_k \omega$ is in $W_d^{\infty,\infty}(\R^n,\bigwedge^m)$ and therefore gives rise to an additive cochain $X^{\omega^k}\in \cW_{\infty,\infty}(\poly{m})$ with $\| X^{\omega^k}\|_{\infty,\infty}=\| \omega^k\|_{\infty,\infty}$, by the paragraph above. Let $T\in\poly{m}$ and let $k$ be large enough so that $T$ is supported in $B(0,k-2)$. By Fuglede's lemma and Lemma \ref{lem:modulus translates}, then $X^{\omega^k}({\varphi_x}_\#T)=X^{\omega^\ell}({\varphi_x}_\# T)$ for all $\ell\geq k$ and almost every $x\in B(0,1)$. By Theorem \ref{thm:smoothening}, we obtain that $X^{\omega^k}(T)=X^{\omega^\ell}(T)$ for every $T$ supported in $B(0,k-2)$ and every $\ell\geq k$. We can thus define $X^\omega(T):=\lim_{k\to\infty} X^{\omega^k}(T)$ for all $T\in\poly{m}$. This clearly yields an additive cochain in $\cW_{\infty,\infty}(\poly{m})$. Since $\| \omega^k\|_{\infty,\infty} \to \| \omega\|_{\infty,\infty}$ we clearly get that $\|X^\omega\|_{\infty,\infty} \leq \| \omega\|_{\infty,\infty}$. We eventually have that $\| X^\omega\|_{\infty,\infty}=\| \omega\|_{\infty,\infty}$ by the same proof as in Section \ref{subsec:construction cochain}.

\begin{remark}
We note here that Remark \ref{rk:cochain from form} remains true in the case $q=p=\infty$.
\end{remark}


 \section{From Sobolev cochains to Sobolev differential forms}
 \label{sec:cochain}
In this section, we construct a continuous linear map $$\Phi_m:W_{q,p}(\poly{m})\to W_d^{q,p}(\R^n,\Lambda^m)$$ for all $1\leq m\leq n$ and $1<p,q\leq \infty$. Note that no restrictions will be put on $p$ and $q$ other than $p,q>1$.
%

Let $X \in \cW_{q,p}(\poly{m})$ be an additive cochain with upper norm $h\in L^q(\R^n)$ and upper gradient $g\in L^p(\R^n)$. For $y\in \R^n$ and $\alpha\in \Lambda(m,n)$, define a map $\varphi_{\alpha,y}:\R^m \to \R^n$ by
$$\varphi_{\alpha,y}(x)=y+\sum_{i=1}^m x_i e_{\alpha(i)}.$$
Fix $\alpha\in \Lambda(m,n)$ and write $\R^n=V_\alpha + V_\alpha^{\perp}$, where $V_\alpha=\span \{ (e_{\alpha(i)})_i \}$ and where $V_\alpha^\perp$ denotes the orthogonal complement of $V_\alpha$. Moreover, fix $y\in V_\alpha^{\perp}$ such that $\| h\circ \varphi_{\alpha,y}\|_q <\infty$. The coefficient of the differential form in the direction $\alpha$ will be defined at almost every point of the $m$-plane $\varphi_{\alpha,y}(\R^m)$. 

Let $\mathcal{S}_{bs}(\R^m)$ be the space of simple functions $\xi$ on $\R^m$ such that the level sets of $\xi$ are (bounded) polyhedral sets in $\R^m$. Then $\mathcal{S}_{bs}(\R^m)$ is a vector subspace of $L^{q'}(\R^m)$,  where $q'\in[1, \infty)$ is such that $\frac{1}{q}+\frac{1}{q'}=1$.
Define
 $$\begin{array}{ccc}
 \xi: \mathcal{S}_{bs}(\R^m) & \to & \overline\R \\
 \theta & \mapsto & X({\varphi_{\alpha,y}}_\#\dblebracket{\theta}), 
 \end{array}$$
 where $\dblebracket{\theta}$ denotes the polyhedral $m$-chain in $\R^m$ induced by the simple function $\theta$.
 We have
 \begin{eqnarray*}
 |\xi(\theta)| &\leq& \int_{\R^n} h(x) \dd{{\varphi_{\alpha,y}}_\# \dblebracket{\theta}}(x)\\
 &=& \int_{\R^m} h(\varphi_{\alpha,y}(z)) \dd{\dblebracket{\theta}}(z)\\
 &=&\int_{\R^m} |\theta| \cdot h\circ \varphi_{\alpha,y} \ d\lm^m\\
 &\leq& \| \theta \|_{q'} \cdot \| h\circ \varphi_{\alpha,y}\|_q <\infty.
 \end{eqnarray*}
 Thus the function $\xi$ has values in $\R$ and is hence additive (note that additivity property in Definition~\ref{def:cochain} only applies when all terms are finite). It follows that $\xi$ is $\mathbb{Q}$-linear and thus, by the above inequality, that $\xi$ is $\R$-linear. By the Hahn-Banach extension theorem, there exists a continuous linear functional $\overline\xi:L^{q'}(\R^m)\to\R$ such that $\overline{\xi}|_{\mathcal{S}_{bs}(\R^m)}=\xi$ and $\| \overline\xi \|_{(L^{q'})^*} \leq \| h\circ \varphi_{\alpha,y}\|_q$. Then there exists $\lambda\in L^q(\R^m)$ such that
 $$\overline\xi(\theta)=\int_{\R^m}\lambda \cdot \theta\ d{\lm^m}$$
and, in particular,
 \begin{equation} \label{eq:differentiation}
 \frac{1}{r^m} X({\varphi_{\alpha,y}}_\# \dblebracket{ \chi_{z+[0,r]^m}}) = \frac{1}{r^m} \int_{z+[0,r]^m} \lambda\ d{\lm^m}
 \end{equation}
 for every $z\in\R^m$ and all $r>0$.
By the Lebesgue differentiation theorem, the limit as  $r\to 0^+$ of the quantity in (\ref{eq:differentiation}) exists for almost every $z\in\R^m$. For such $z$, define
 $$\omega^X(\varphi_{\alpha,y}(z),\alpha):=\lim_{r\to 0^+} \frac{1}{r^m} X({\varphi_{\alpha,y}}_\# \dblebracket{ \chi_{z+[0,r]^m}}).$$
 Consequently, $\omega^X(x,\alpha)$ exists for almost every $x \in \R^n$, 
 $$\omega^X(x,\alpha)=\lim_{r\to 0^+} \frac{1}{r^m} X({\varphi_{\alpha,x}}_\# \dblebracket{ \chi_{[0,r]^m}})$$
 and since
 $$\frac{1}{r^m} \left|X({\varphi_{\alpha,x}}_\# \dblebracket{ \chi_{[0,r]^m}})\right|\leq \frac{1}{r^m} \int_{[0,r]^m} h\circ \varphi_{\alpha,x} \,d{\lm^m},$$
 we have
 \begin{equation}  \label{eq:norm bound diff form}
 |\omega^X(x,\alpha)| \leq h(x)
 \end{equation}
 for almost every $x\in\R^n$.
 
If $m\leq n-1$ we can define similarly 
 \begin{equation}   \label{eq:d-omega-beta}
 d\omega^X(x,\beta):= \lim_{r\to 0^+} \frac{1}{r^{m+1}} X({\varphi_{\beta,x}}_\# \bdry \dblebracket{\chi_{[0,r]^{m+1}}})
 \end{equation}
 for every $\beta\in \Lambda(m+1,n)$ and for almost every $x\in \R^n$.
 It follows as above that 
 \begin{equation}  \label{eq:norm bound derivative diff form}
 |d\omega^X(x,\beta)|\leq g(x)
 \end{equation}
 for almost every $x\in\R^n$.
 Define, for almost every $x\in \R^n$,
 \begin{equation}\label{eq:def-omega}
  \omega^X(x):=\sum_{\alpha\in \Lambda(m,n)} \omega^X(x,\alpha)dx^\alpha
 \end{equation}
 and
 \begin{equation}\label{eq:def-d-omega}
  d\omega^X(x):=\sum_{\beta\in \Lambda(m+1,n)} d\omega^X(x,\beta)dx^\beta.
 \end{equation}
 
It follows from the definition of $\omega^X$ and inequality \eqref{eq:norm bound diff form} that an additive cochain in the same equivalence class as $X$ yields a differential form which is in the same equivalence class as $\omega$. We now prove the following:

\begin{lemma}\label{lem:distr-deriv-omega-X}
 If $m\leq n-1$ then the distributional exterior derivative of the $m$-form $\omega^X$ defined in \eqref{eq:def-omega} is given by the $(m+1)$-form $d\omega^X$ defined in \eqref{eq:def-d-omega}, that is,
$$\int_{\R^n} d\omega^X\wedge \nu = (-1)^{m+1} \int_{\R^n} \omega^X \wedge d\nu$$
for every smooth compactly supported $(n-m-1)$-form $\nu$.
\end{lemma}

This lemma shows that $\omega^X$ is in $W_d^{q,p}(\R^n,\bigwedge^m)$ and hence the map $$\Phi_m:W_{q,p}(\poly{m})\to W_d^{q,p}(\R^n,\wedge^m)$$ given by $\Phi_m(X)=\omega^X$ is linear and satisfies
$$\| \Phi_m(X) \|_{q,p}\leq C \| X \|_{q,p}$$ 
for all $X\in W_{q,p}(\poly{m})$, where $C>0$ depends only on $m$ and $n$.

\begin{proof}[Proof of Lemma~\ref{lem:distr-deriv-omega-X}]
Fix a smooth compactly supported simple $(n-m-1)$-form $\nu=f dx^\gamma$, $\gamma\in \Lambda(n-m-1,n)$, and let $\beta\in \Lambda(m+1,n)$  be such that $$e_{\beta(1)}\wedge \cdots \wedge e_{\beta(m+1)}\wedge e_{\gamma(1)}\wedge \cdots \wedge e_{\gamma(n-m-1)}=(-1)^k e_1\wedge \cdots \wedge e_n$$ for some $k$.
For $i=1,\dots,m+1$, let $\alpha^i\in \Lambda(m,n)$ be such that $\alpha^i(j)=\beta(j)$ for $j\in \{ 1,\dots,i-1\}$ and $\alpha^i(j)=\beta(j+1)$ for $j\in \{ i,\dots,m\}$.
With this notation, we have $dx^{\alpha^i}\wedge dx^{\beta(i)}=(-1)^{m-i+1}dx^\beta$, thus $$dx^{\alpha^i}\wedge dx^{\beta(i)} \wedge dx^\gamma=(-1)^{k+m-i+1} dx^1\wedge\dots\wedge dx^n$$ and therefore
$$\int_{\R^n} \omega^X \wedge d\nu= \sum_{i=1}^{m+1} (-1)^{k+m-i+1} \int_{\R^n} \omega^X(x,\alpha^i) \frac{\partial f}{\partial x_{\beta(i)}}(x) \,dx.$$
We can now write 
\begin{equation*}
\begin{split}
\int_{\R^n}& \omega^X \wedge d\nu\\
&=\sum_{i=1}^{m+1} (-1)^{k+m-i+1} \int_{\R^n} \lim_{r\to 0^+} \frac{X({\varphi_{\alpha^i,x}}_\# \dblebracket{\chi_{[0,r]^m}})}{r^m} \cdot \frac{\partial f}{\partial x_{\beta(i)}}(x)\, dx\\
&=\sum_{i=1}^{m+1} (-1)^{k+m-i+1} \int_{\R^n} \lim_{r\to 0^+} \frac{X({\varphi_{\alpha^i,x}}_\# \dblebracket{\chi_{[0,r]^m}})}{r^m} \cdot \frac{f(x)-f(x-re_{\beta(i)})}{r} \,dx\\
&=\sum_{i=1}^{m+1} (-1)^{k+m-i+1}\lim_{r\to 0^+}  \int_{\R^n}  \left(\frac{X({\varphi_{\alpha^i,x}}_\# \dblebracket{\chi_{[0,r]^m}})}{r^{m+1}} - \frac{X({\varphi_{\alpha^i,x+re_{\beta(i)}}}_\#  \dblebracket{\chi_{[0,r]^m}})}{r^{m+1}} \right) f(x)\,dx.\\
\end{split}
\end{equation*}
Here, the limit can be taken outside the integral by the Lebesgue dominated convergence theorem and the maximal function theorem. Indeed,  writing $x\in\R^n$ uniquely as $x= y+z$ with $y\in V_\alpha^\perp$ and $z\in V_\alpha= \varphi_{\alpha, 0}(\R^m)$ we have
$$
\left| \frac{X({\varphi_{\alpha^i,x}}_\# \dblebracket{\chi_{[0,r]^m}})}{r^m} \right| \leq C M(h\circ\varphi_{\alpha, y})(\varphi_{\alpha, 0}^{-1}(z)) 
$$
for all $r >0$, where $M(h\circ\varphi_{\alpha, y})$ is the Hardy-Littlewood maximal function of $h\circ\varphi_{\alpha, y}$ and where $C$ is a constant only depending on $m$. If $q<\infty$ then, by the maximal function theorem and Fubini theorem, we have
$$\int_{V_\alpha^\perp}\int_{V_\alpha}M(h\circ\varphi_{\alpha, y})^q(\varphi_{\alpha, 0}^{-1}(z))\,dz\,dy \leq C \int_{V_\alpha^\perp}\int_{\R^m} h^q\circ\varphi_{\alpha, y}(z)\,dz\,dy = C \|h\|_q^q<\infty$$ for some constant $C$ and so the map $x=y+z\mapsto M(h\circ\varphi_{\alpha, y})(\varphi_{\alpha, 0}^{-1}(z))$ is in $L^q(\R^n)$. If $q=\infty$ then we obtain similarly that the map $x=y+z\mapsto M(h\circ\varphi_{\alpha, y})(\varphi_{\alpha, 0}^{-1}(z))$ is in $L^\infty(\R^n)$.

Now, observe that
$${\varphi_{\beta,x}}_\# \bdry\dblebracket{\chi_{[0,r]^{m+1}}}=\sum_{i=1}^{m+1} (-1)^i \left( {\varphi_{\alpha^i,x}}_\#\dblebracket{\chi_{[0,r]^m}}-{\varphi_{\alpha^i,x+re_{\beta(i)}}}_\#\dblebracket{\chi_{[0,r]^m}} \right).$$
This shows that
$$\int_{\R^n} \omega^X \wedge d\nu=(-1)^{k+m+1} \lim_{r\to 0} \int_{\R^n}  \frac{X({\varphi_{\beta,x}}_\#\bdry\dblebracket{\chi_{[0,r]^{m+1}}})}{r^{m+1}} f(x) \, dx=(-1)^{m+1} \int_{\R^n} d\omega^X\wedge \nu, $$
 where we again use dominated convergence as above, replacing the upper norm $h$ with the upper gradient $g$.
 \end{proof}

\begin{remark}   \label{rk:form from cochain}
From the construction of the differential form $\omega^X$, it follows that for all $\alpha\in \Lambda(m,n)$ and all $y\in V_\alpha^\perp$ such that $\| h\circ \varphi_{\alpha,y} \|_q<\infty$, the following holds. For all $T\in\poly{m}$ such that $\spt(T) \subset \varphi_{\alpha,y}(\R^m)$,
 \begin{equation}  \label{eq:omega(T)=T(hatomega)}
X(T)=\int_T \omega^X.
 \end{equation}
 In particular, for all $T\in\polypar{m}$ and for almost every $x\in\R^n$,
$$X({\varphi_x}_\# T)=\int_{{\varphi_x}_\# T} \omega^X,$$
where $\varphi_x$ is defined by $\varphi_x(y)=x+y$. Similarly, for all $S\in\polypar{m+1}$ and for almost every $x\in\R^n$,
$$dX({\varphi_x}_\# S)=\int_{{\varphi_x}_\# S} d\omega^X.$$
\end{remark} 
 

\section{Proof of Theorems~\ref{thm:main result} and \ref{thm:main-result-closed}} \label{sec:proof-main-thm}

We first give the proof of Theorem~\ref{thm:main result}.

\begin{proof}[Proof of Theorem~\ref{thm:main result}]
Let $\Psi_m$ and $\Phi_m$ be the continuous linear maps constructed in Sections~\ref{sec:form} and \ref{sec:cochain}, respectively. By Proposition~\ref{prop:Psi-m-isometric}, the map $\Psi_m$ is isometric. In order to show that $\Psi_m$ is surjective it thus suffices to show that $\Psi_m \circ \Phi_m$ is the identity.
For this fix a cochain $Y \in W_{q,p}(\poly{m})$ and set $Z=\Psi_m(\Phi_m(Y))$. Let $X\in \cW_{q,p}(\poly{m})$ be an additive cochain which is a representative of $Y-Z$. We will show that $\| X\|_{q,p}=0$ and thus that $X$ is zero as an element of $W_{q,p}(\poly{m})$. By Remarks \ref{rk:cochain from form} and \ref{rk:form from cochain}, we know that for all $T\in\polypar{m}$, all $S\in\polypar{m+1}$, and almost every $x\in\R^n$,
$$X({\varphi_x}_\# T)=0 \text{ and } dX({\varphi_x}_\# S)=0.$$
In particular, $X_r$ is zero on $\polypar{m}$ and $dX_r$ is zero on $\polypar{m+1}$ for every $r>0$. We next show that $X_r$ is in fact zero on $\poly{m}$ for every $r>0$. 

For this, fix $T\in\poly{m}$ and let $\varepsilon>0$. The deformation theorem \cite[4.2.9]{Fed69} asserts that there exist $T'\in\polypar{m}$, $S\in\poly{m+1}$ and $R\in\poly{m}$ such that 
$$T=T'+R+\bdry S, $$ 
and 
\begin{equation}   \label{eq:deformation theorem}
\mass(S)\leq \gamma\varepsilon \mass(T), \quad \mass(R)\leq \gamma\varepsilon \mass(\bdry T), 
\end{equation}
where $\gamma$ is a constant only depending on $n$ and $m$. Let $r>0$.
Since $X_r(T')=0$ it follows that 
$$|X_r(T)|\leq |X_r(R)|+|X_r(\bdry S)|.$$
Now recall that $X_r\in W_{\infty,\infty}(\poly{m})$ with upper norm $h_r:=\fint_{B(0,r)}h(\cdot+y) dy$ and upper gradient $g_r:=\fint_{B(0,r)} g(\cdot+y)dy$.
The upper norm and upper gradient inequalies together with (\ref{eq:deformation theorem}) yield
$$|X_r(R)|\leq \| h_r\|_\infty \cdot \mass(R) \leq \gamma \varepsilon \| h_r\|_\infty \mass(\bdry T)$$ 
and 
$$|X_r(\bdry S)|\leq \|g_r\|_\infty \cdot \mass(S) \leq \gamma \varepsilon \|g_r\|_\infty \mass(T).$$
Letting $\varepsilon \to 0$ we obtain $X_r(T)=0$ for every $T\in\poly{m}$. This shows that $X_r$ is zero on $\poly{m}$ for every $r>0$. It thus follows from Theorem~\ref{thm:smoothening} that $\| X\|_{q,p}=0$ since $p >n-m$ or $q\leq pn/(n-p)$.   
 \end{proof}
 
 \begin{remark}
 The proof above applies word by word in the case that $q=p=\infty$. Indeed, by Sections \ref{sec:construction Psi case infinity}, the map $\Psi_m$ is well defined and isometric also in this case. Moreover, no finiteness conditions on $q$ and $p$ were placed in the construction of $\Phi_m$ and in the statement of Theorem \ref{thm:smoothening}.
 \end{remark}
 
Corollary~\ref{cor:Wolfe's theorem} now comes as a direct consequence of the remark above together with the following lemma. First recall that by definition the space of flat $m$-cochains in $\R^n$ is the dual space of the space $\mathcal{F}_m$ of flat $m$-chains in $\R^n$.
 
\begin{lemma}   \label{lem:flat cochains}
The space of flat $m$-cochains in $\R^n$ is isometrically isomorphic to $W_{\infty,\infty}(\poly{m})$.
\end{lemma}

\begin{proof}
Let $X\in W_{\infty,\infty}(\poly{m})$ be a cochain with upper norm $h\in L^\infty(\R^n)$ and upper gradient $g\in L^\infty(\R^n)$. Fix $T\in\poly{m}$. Let $R\in\poly{m}$ and $S\in\poly{m+1}$ be such that $T = R + \partial S$ and let $r>0$. Since $|X_r(R)|\leq \|h\|_\infty \mass(R)$ and $|X_r(\partial S)|\leq \|g\|_\infty \mass(S)$ it follows that 
\begin{equation*}
 |X_r(T)|\leq \max\{\|h\|_\infty, \|g\|_\infty\} (\mass(R) + \mass(S)).
\end{equation*}
 Since $h$, $g$, $R$, and $S$ were arbitrary it follows that $$|X_r(T)|\leq \|X\|_{\infty, \infty} |T|_\flat$$ for every $T\in\poly{m}$ and every $r>0$. By Theorem~\ref{thm:smoothening} (with $q=p=\infty$), $X_r(T)$ converges to $X(T)$ for every $T$ and hence $$|X(T)|\leq \|X\|_{\infty, \infty} |T|_\flat.$$ Since $\poly{m}$ is dense in $\mathcal{F}_m$ it follows that there exists a unique extension to a continuous linear functional $\bar{X}:\mathcal{F}_m\to\R$ satisfying $\mathcal{F}(X)\leq \|X\|_{\infty,\infty}$, where $\mathcal{F}(X)$ is the dual norm to the flat norm $| \cdot |_\flat$.

Conversely, let $X$ be a flat $m$-cochain and set $h=g=\mathcal{F}(X)$. It follows that for every $T\in\poly{m}$ and every $S\in\poly{m+1}$ we have
$$|X(T)|\leq \mathcal F(X)\cdot |T|_\flat\leq \mathcal F(X) \cdot \mass (T) = \int_{\R^n} h\,d\|T\|$$
and
$$|X(\bdry S)|\leq \mathcal F(X)\cdot |\bdry S|_\flat\leq \mathcal F(X)\cdot \mass (S) = \int_{\R^n} g\,d\|S\|,$$
which shows that $h$ and $g$ are upper norm and upper gradient of $X$, respectively. Therefore, the restriction of $X$ to $\poly{m}$ is a cochain in  $W_{\infty,\infty}(\poly{m})$ and $\| X \|_{\infty,\infty} \leq \mathcal F(X)$. Since these two maps are clearly inverses of each other the proof is complete.
\end{proof}

 We finally turn to the proof of Theorem \ref{thm:main-result-closed}. This theorem is a direct consequence of the following three lemmas. Fix $1\leq m\leq n-1$ as well as $s>n-m+1$ and $1<p<\infty$.

 \begin{lemma}
 The space $W_p(\poly{m}^0)$ is isometrically isomorphic to the space 
\begin{equation}\label{eq:def-space-V}
 V:= \{ Y\in W_{p,s}(\poly{m+1}): dY=0 \}.
 \end{equation}
 \end{lemma} 
 
 Note that the norm $\| Y\|_{p,s}$ of an element $Y\in V$ is independent of the value of $s$ and that $dY=0$ means equality as an element in $W_{s,s}(\poly{m+2})$.
 
 \begin{proof}
 We construct a linear isometric map $\varrho: W_p(\poly{m}^0) \to V$
 as follows. Let $X\in\cW_p(\poly{m}^0)$ be an additive cochain and define a function $\varrho(X):\poly{m+1}\to \overline\R$ by $\varrho(X)(S):=X(\bdry S)$. This is clearly an additive cochain, and a non-negative Borel function is an upper gradient of $X$ if and only if it is an upper norm of $\varrho(X)$. Moreover, the constant function zero is an upper gradient of $\varrho(X)$. This shows that $\varrho(X)$ is in $\cW_{p,s}(\poly{m+1})$ and $d\varrho(X)=0$ everywhere. It follows easily that if $X, X'\in \cW_p(\poly{m}^0)$ belong to the same equivalence class, then $\varrho(X)$ and $\varrho(X')$ are equivalent as elements of $\cW_{p,s}(\poly{m+1})$. Thus, $\varrho$ is well-defined as a map from $W_p(\poly{m}^0)$ to $V$ and is clearly linear and isometric.
 
 It remains to show that $\varrho$ is surjective. For this, let $Y\in\cW_{p,s}(\poly{m+1})$ be an additive cochain such that $dY=0$ as an element of $\cW_{s,s}(\poly{m+2})$. For $r>0$, let $Y_r$ be the averaged cochain and note that 
 \begin{equation}  \label{eq:average Y}
 Y_r(S+S')=Y_r(S)+Y_r(S')
 \end{equation}
 for all $S,S'\in\poly{m+1}$. Since $(dY)_r \equiv 0$, it follows that $Y_r(S)=Y_r(S')$ for all $S$ and $S'$ for which $\bdry S= \bdry S'$. Define a function $X:\poly{m}^0 \to \overline \R$ as follows. Let $T\in\poly{m}^0$ and let $S\in\poly{m+1}$ be any element with $\bdry S=T$. Define $X(T):= \lim_{r\to 0} Y_r(S)$ if the limit exists, and $X(T)=\infty$ otherwise. Note that the existence and the value of the limit is independent of the choice of $S$ by the remark above. It follows directly from the definition and from \eqref{eq:average Y} that $X$ is an additive cochain. We now show that $X$ has an upper gradient in $L^p(\R^n)$. Indeed, by Theorem \ref{thm:smoothening}, there exists $\Lambda \subset \poly{m+1}$ such that $M_p(\Lambda)=0$ and such that 
 $$| Y_r(S)-Y(S)|\to 0 \text{ as } r\to 0$$
 for every $S\in \poly{m+1} \setminus \Lambda$. Since $M_p(\Lambda)=0$, there exists a non-negative Borel function $f\in L^p(\R^n)$ such that 
 $$\int_{\R^n} f d\| S\|=\infty$$
 for every $S\in\Lambda$. Now let $h\in L^p(\R^n)$ be an upper norm of $Y$. It is easy to see that $h+f$ is an upper gradient of $X$. Indeed, let $T\in\poly{m}^0$ and $S\in\poly{m+1}$ such that $\bdry S=T$. If $S\not\in \Lambda$, then $|Y_r(S)-Y(S)|\to 0$ and thus $X(T)=Y(S)$ and
 $$|X(T)| = |Y(S)| \leq \int_{\R^n} h\, d\| S\|.$$
 If $S\in \Lambda$, then
 $$|X(T)|\leq \infty=\int_{\R^n} f\, d\| S\|.$$
 This shows that $h+f$ is an upper gradient of $X$ and thus $X\in\cW_p(\poly{m}^0)$, as claimed.
 Finally, for every $S\in\poly{m+1}\setminus \Lambda$, we have that
 $$\varrho(X)(S)=X(\bdry S)=Y(S)$$
 by the above and, therefore, $\varrho(X)=Y$ as elements of $W_{p,s}(\poly{m+1})$. This shows that $\varrho$ is surjective.
 \end{proof}

 \begin{lemma}
 The space $V$ defined in \eqref{eq:def-space-V} is isometrically isomorphic to 
 \begin{equation}\label{eq:def-space-U}
  U:=\left\{ \omega\in W_d^{p,s}\left(\R^n,\bigwedge\nolimits^{m+1}\right): d\omega=0 \right\}.
 \end{equation}
 \end{lemma}
  
  \begin{proof}
 We define a linear isometric map from $U$ to $V$ as follows. Given $\omega\in U$ set $X^\omega=\Psi_{m+1}(\omega)$, where $\Psi_{m+1}$ is the linear isometric map defined in Proposition \ref{prop:Psi-m-isometric}. It follows that $X^\omega \in W_{p,s}(\poly{m+1})$. Moreover, if $m<n-1$, then \eqref{eq:psi-commutes-d} shows that $dX^\omega=\Psi_{m+2}(d\omega)=0$ as an element of $W_{s,s}(\poly{m+2})$. If $m=n-1$, then $dX^\omega=0$ trivially. In particular, we have that $X^\omega\in V$. Since $\Psi_{m+1}$ is linear and isometric, it follows that the map $\omega \mapsto X^\omega$ is linear and isometric. It remains to show that this map is surjective. For this, let $Y\in W_{p,s}(\poly{m+1})$ with $dY=0$. Define $\omega^Y=\Phi_{m+1}(Y)$, where $\Phi_{m+1}$ is the map defined in Section \ref{sec:cochain}. Since $dY=0$, it follows from \eqref{eq:d-omega-beta} that $d\omega^Y=0$ and hence $\omega^Y\in U$. Since $\Psi_{m+1}\circ \Phi_{m+1}$ is the identity, see the proof of Theorem \ref{thm:main result}, it follows that $X^{\omega^Y}=Y$. This shows that the map $\omega \mapsto X^\omega$ is surjective.
  \end{proof}

 \begin{lemma}
 The space $U$ defined in \eqref{eq:def-space-U} is isometrically isomorphic to the space $\overline{W}_d^{p}(\R^n,\bigwedge^{m})$.
 \end{lemma}
 
 \begin{proof}
 Clearly, the map $\overline{W}_d^{p}(\R^n,\bigwedge^{m}) \to U$ given by $[\omega ] \mapsto d\omega$ is well-defined, linear, and isometric. In order to show that it is surjective, let $\nu \in U$. Let 
 $$T:L^p_{\loc}\left(\R^n,\bigwedge\nolimits^{m+1}\right) \to L^p_{\loc}\left(\R^n,\bigwedge\nolimits^{m}\right)$$
 be the chain homotopy operator defined in \cite{IL93} and set $\omega=T(\nu)$. The coefficients of $\omega$ are in $L^p_{\loc}(\R^n)$ and thus, in particular, in $L^1_{\loc}(\R^n)$. By the chain homotopy formula \cite[Lemma 4.2]{IL93}, $T(\nu)$ has a distributional exterior derivative $dT(\nu)$ in $L^p_{\loc}\left(\R^n,\bigwedge^{m+1}\right)$ and
 $$\nu=T(d\nu)+dT(\nu)=dT(\nu).$$
 The last equality is a consequence of the fact that $d\nu=0$.
 It follows that $dT(\nu)\in L^p \left(\R^n,\bigwedge^{m+1}\right)$ and hence $T(\nu)\in W_{d,\loc}^{1,p}(\R^n,\bigwedge^{m})$. Since $dT(\nu)=\nu$, this shows that the map $[\omega] \mapsto d\omega$ is surjective.
 \end{proof}
 

\end{document}